\documentclass{amsart}

\usepackage{amssymb}
\usepackage{amsmath}
\usepackage{graphicx}
\usepackage{amsthm}

\newcommand{\erase}[1]{}

\newtheorem{theorem}{Theorem}[section]
\newtheorem{lemma}[theorem]{Lemma}
\newtheorem{proposition}[theorem]{Proposition}
\newtheorem{corollary}[theorem]{Corollary}

\newtheorem{_definition}[theorem]{Definition}
\newenvironment{definition}{\begin{_definition}\rm}{\end{_definition}}

\newtheorem{_remark}[theorem]{\it Remark}
\newenvironment{remark}{\begin{_remark}\rm}{\end{_remark}}

\newtheorem{_example}[theorem]{Example}
\newenvironment{example}{\begin{_example}\rm}{\end{_example}}

\numberwithin{equation}{section}
\numberwithin{table}{section}

\newcommand{\E}{\mathbb{E}}
\newcommand{\F}{\mathord{\mathbb F}}

\renewcommand{\P}{\mathord{\mathbb  P}}
\newcommand{\Q}{\mathord{\mathbb  Q}}
\newcommand{\R}{\mathord{\mathbb R}}
\newcommand{\Z}{\mathord{\mathbb Z}}

\newcommand{\DDD}{\mathord{\mathcal D}}
\newcommand{\EEE}{\mathord{\mathcal E}}
\newcommand{\FFF}{\mathord{\mathcal F}}
\newcommand{\GGG}{\mathord{\mathcal G}}
\newcommand{\HHH}{\mathord{\mathcal H}}
\newcommand{\MMM}{\mathord{\mathcal M}}
\newcommand{\PPP}{\mathord{\mathcal P}}
\newcommand{\RRR}{\mathord{\mathcal R}}

\newcommand{\inj}{\hookrightarrow}
\newcommand{\surj}{\mathbin{\to \hskip -7pt \to}}
\newcommand{\isom}{\mathbin{\,\raise -.6pt\rlap{$\to$}\raise 3.5pt \hbox{\hskip .3pt$\mathord{\sim}$}\,}}

\newcommand{\set}[2]{\{\; {#1} \; \mid \; {#2} \;  \}}
\newcommand{\shortset}[2]{\{ {#1} \,|\, {#2}   \}}

\newcommand{\gen}[1]{\langle {#1}  \rangle}

\newcommand{\tensor}{\otimes}

\newcommand{\sprime}{\sp\prime}

\newcommand{\spcirc}{\sp{\mathord{\circ}}}
\newcommand{\sptimes}{\sp{\times}}
\newcommand{\sperp}{\sp{\perp}}

\newcommand{\dual}{\sp{\vee}}

\newcommand{\fiberproduct}{\Box}
\newcommand{\inv}{\sp{-1}}

\newcommand{\PGU}{\mathord{\mathrm {PGU}}}

\newcommand{\PGL}{\mathord{\mathrm {PGL}}}
\newcommand{\OG}{\mathord{\mathrm {O}}}

\newcommand{\id}{\mathord{\mathrm {id}}}

\newcommand{\Ker}{\operatorname{\mathrm {Ker}}\nolimits}

\newcommand{\Aut}{\operatorname{\mathrm {Aut}}\nolimits}
\newcommand{\Gal}{\operatorname{\mathrm {Gal}}\nolimits}

\newcommand{\Grass}{\mathord{\mathrm {Grass}}}

\newcommand{\rmand}{\textrm{and}}

\newcommand{\quand}{\quad\rmand\quad}

\newcommand{\typeI}{{\rm I}}

\newcommand{\intM}[2]{\langle #1\rangle_{#2}}

\newcommand{\NE}{\mathord{{\rm Nef}}}
\newcommand{\DR}{\mathord{{\rm DR}}}
\newcommand{\aut}{\mathord{{\it Aut}}}

\newcommand{\Xps}{X_{p, \sigma}}
\newcommand{\Xpss}{X_{p, \sigma\sprime}}
\newcommand{\Sps}{S_{p, \sigma}}
\newcommand{\Spss}{S_{p, \sigma\sprime}}

\newcommand{\sphyp}{^*}
\newcommand{\Tiota}{T_{\iota}}
\newcommand{\switch}{\mathord{\mathrm{sw}}}
\newcommand{\Frob}{\mathord{\mathrm{Fr}}}
\newcommand{\Cremona}{\mathord{\mathrm{Cr}}}

\newcommand{\FQ}{\mathord{\mathrm{FQ}}}
\newcommand{\Lts}{L_{26}}
\newcommand{\Xpone}{X_{p, 1}}
\newcommand{\Xpten}{X_{p, 10}}

\newcommand{\Xtwoone}{X_{2, 1}}
\newcommand{\Xtwoten}{X_{2, 10}}
\newcommand{\Xthreeone}{X_{3, 1}}
\newcommand{\Xthreeten}{X_{3, 10}}
\newcommand{\Stwoone}{S_{2, 1}}
\newcommand{\Stwoten}{S_{2, 10}}
\newcommand{\Sthreeone}{S_{3, 1}}
\newcommand{\Sthreeten}{S_{3, 10}}
\newcommand{\POG}{\mathord{\mathrm{ PO}}}
\newcommand{\Isot}{\mathord{\mathrm{Isot}}}
\newcommand{\Gen}{\mathord{\mathrm{Gen}}}

\begin{document}

\title [Supersingular K3 surfaces] {On certain duality of  N\'eron-Severi lattices of 
 supersingular $K3$ surfaces}

\author{Shigeyuki Kond\=o}
\email{kondo@math.nagoya-u.ac.jp}
\address{Graduate School of Mathematics, Nagoya University, Nagoya,
464-8602, JAPAN}
\thanks{The first author was partially supported by
JSPS Grant-in-Aid  for Scientific Research (S)
  No.22224001.}

\author{Ichiro Shimada}
\email{shimada@math.sci.hiroshima-u.ac.jp}
\address{Department of Mathematics, 
Graduate School of Science, 
Hiroshima University, 
Higashi-Hiroshima, 
739-8526,  JAPAN}
\thanks{The second author was partially supported by JSPS Grants-in-Aid for Scientific Research (B) No.20340002.}


\begin{abstract}
Let $X$ and $Y$ be  supersingular $K3$ surfaces defined over an algebraically closed field.
Suppose that the sum of their Artin invariants is $11$.
Then there exists a certain duality between their N\'eron-Severi lattices.
We investigate geometric consequences of this duality.
As an application, 
we classify  genus one fibrations on supersingular $K3$ surfaces with Artin invariant $10$
in characteristic $2$ and $3$, and 
give a set of generators of the automorphism group 
of the nef cone of these supersingular $K3$ surfaces.
The difference between 
the automorphism group of a  supersingular $K3$ surface $X$  and 
the automorphism group of its nef cone 
is determined by the period of $X$.
We define the notion of genericity for supersingular $K3$ surfaces 
in terms of the period,
and prove the existence of generic supersingular $K3$ surfaces
in odd characteristics 
for each Artin invariant larger than $1$.
\end{abstract}
\maketitle

\section{Introduction}
A $K3$ surface $X$ defined over an algebraically closed field $k$ is said to be
\emph{supersingular} (in the sense of Shioda) if the rank of its N\'eron-Severi lattice $S_X$ is $22$.
Supersingular $K3$ surfaces exist only when the base field $k$ is of positive characteristic.
Let $X$ be a supersingular $K3$ surface in characteristic $p>0$.
Artin~\cite{MR0371899} proved that 
the discriminant group of  $S_X$ 
is a $p$-elementary abelian group of rank $2\sigma$, 
where $\sigma$ is a positive  integer such that $\sigma\le 10$.
This integer $\sigma$ is called the \emph{Artin invariant} of $X$.
The isomorphism class of the lattice $S_X$ depends only on $p$ and $\sigma$
(Rudakov and Shafarevich~\cite{MR633161}).  
Moreover
supersingular $K3$ surfaces with Artin invariant $\sigma$ form a $(\sigma -1)$-dimensional family,
and a supersingular $K3$ surface with Artin invariant $1$ 
in characteristic $p$ 
is unique up to isomorphisms
(Ogus~\cite{MR563467}, \cite{MR717616}, Rudakov and Shafarevich~\cite{MR633161}).
\par
\medskip
Recently many studies of supersingular $K3$ surfaces in small characteristics with  Artin invariant $1$
have appeared.
For example, for $p=2$, 
Dolgachev and  Kondo~\cite{MR1935564},
Katsura and Kondo \cite{KatsuraKondochar2},
Elkies and Sch{\"u}tt~\cite{ElkiesSchuett};
for $p=3$, 
%
Katsura and Kondo \cite{MR2862188},
Kondo and Shimada \cite{KondoShimada},
Sengupta~\cite{Sengupta};
and for $p=5$,
Shimada~\cite{shimadapreprintchar5}.
On the other hand, 
geometric properties of supersingular $K3$ surfaces with big Artin invariant are not so much known
(e.g. Rudakov and Shafarevich~\cite{MR508830}, \cite{MR633161}, Shioda~\cite{MR918849}, 
Shimada~\cite{MR2129248}, \cite{MR2036331}).  
\par
\medskip
In this paper, we 
present some methods to investigate  supersingular $K3$ surfaces 
with big Artin invariant  by means of the following simple observation.
Let $\Xps$  be  a supersingular $K3$ surface in characteristic $p$ with Artin invariant $\sigma$,
and let $\Sps$ denote its N\'eron-Severi lattice.
\begin{lemma}\label{lem:main}
Suppose that $\sigma+\sigma\sprime=11$.
Then $\Spss$ is isomorphic to $\Sps\dual (p)$,
where $\Sps\dual (p)$ is the lattice obtained from the dual lattice $\Sps\dual$ of $\Sps$
by multiplying the symmetric bilinear form with $p$.
\end{lemma}
Lemma~\ref{lem:main} is proved in  Section~\ref{sec:NSssK3}.
We use this duality between $\Sps$ and $\Spss$
in the study of   genus one fibrations 
and the automorphism groups of  supersingular $K3$ surfaces.
\par
\medskip
First, we apply Lemma~\ref{lem:main} to the classification of 
 genus one fibrations. 
 Note that the N\'eron-Severi lattice $S_Y$ of a $K3$ surface $Y$ is a hyperbolic lattice.
 The orthogonal group $\OG(S_Y)$ of $S_Y$ contains the stabilizer subgroup $\OG^+(S_Y)$
 of a positive cone of $S_Y\tensor \R$ as a subgroup of index $2$.
\begin{definition}
Let $Y$ be a $K3$ surface, and 
let $\phi: Y\to \P^1$ be a  genus one fibration.
We denote by $f_{\phi}\in S_Y$ the class of a fiber of $\phi$.
Let $\psi: Y\to \P^1$ be another genus one fibration on $Y$.
We say that $\phi$ and $\psi$ are \emph{$\Aut$-equivalent}
if there exist $g\in \Aut (Y)$ and $\bar{g}\in \Aut(\P^1)$
such that $\phi\circ g=\bar{g}\circ \psi$ holds,
while we say that $\phi$ and $\psi$ are \emph{lattice equivalent}
if there exists $g\in \OG^+ (S_Y)$ such that $f_{\phi}^g=f_{\psi}$.
We denote by $\E(Y)$ the set of lattice equivalence classes of genus one fibrations on $Y$,
and by $[\phi]\in \E(Y)$ the lattice equivalence class containing $\phi$.
\end{definition}
Many combinatorial properties of a  genus one fibration $\phi: Y\to \P^1$ 
depend only on 
the lattice equivalence class $[\phi]$.
See Proposition~\ref{prop:depends_only_on_latequiv}.
Moreover,
when $\sigma=10$,
the classification of genus one fibrations by $\Aut$-equivalence seems to be  too fine
(see Remark~\ref{rem:heavy}).
Therefore,
we concentrate upon the study of  lattice equivalence classes.
\par
\medskip
Using Lemma~\ref{lem:main}, we prove the following:
\begin{theorem}\label{thm:genusone}
Suppose that  $\sigma+\sigma\sprime=11$.
Then there exists  a canonical one-to-one 
correspondence  
$[\phi]\mapsto [\phi\sprime]$
between $\E(X_{p, \sigma})$ and $\E(X_{p, \sigma\sprime})$.
\end{theorem}
We say that a  genus one fibration 
is \emph{Jacobian} if it admits a section.
\begin{theorem}\label{thm:nosections1}
Suppose that a genus one fibration $\phi:\Xps\to\P^1$ is a Jacobian fibration,
and let $\phi\sprime :\Xpss\to\P^1$ be a genus one fibration on $\Xpss$ with $\sigma\sprime=11-\sigma$
such that $[\phi\sprime]\in \E(\Xpss)$ corresponds to $[\phi]\in \E(\Xps)$ by Theorem~\ref{thm:genusone}.
Then $\phi\sprime$ does not admit a section.
\end{theorem}
Elkies and Sch\"utt~\cite{ElkiesSchuett} proved the following: 
\begin{theorem}[\cite{ElkiesSchuett}]\label{thm:ESsection}
Any genus one fibration on $\Xpone$ admits a section.
\end{theorem}
Thus we obtain another proof of~\cite[Proposition 12.1]{EvdG}:
\begin{corollary}[\cite{EvdG}]\label{cor:noJacobian}
There exist no Jacobian fibrations on $\Xpten$.
\end{corollary}
By an \emph{$ADE$-type}, we mean a finite formal sum of the symbols 
$A_i$ $(i\ge 1)$, $D_i$ $(j\ge 4)$ and $E_k$ $(k=6,7,8)$
with non-negative integer coefficients.
For a genus one fibration $\phi: Y\to\P^1$ on a $K3$ surface $Y$,
we have the $ADE$-type of reducible fibers of $\phi$.
This $ADE$-type 
depends only on the lattice equivalence class $[\phi]\in \E(Y)$
(see Proposition~\ref{prop:depends_only_on_latequiv}).
Therefore we can use $R_{[\phi]}$ to denote the $ADE$-type of the reducible fibers of $\phi$.
\par 
\medskip
From the classification of lattice equivalence classes of genus one fibrations of
$\Xtwoone$  by Elkies and Sch\"utt~\cite{ElkiesSchuett},  and  that of  $\Xthreeone$ by Sengupta~\cite{Sengupta},
we obtain the classification of lattice equivalence classes of genus one fibrations on
$\Xtwoten$ and $\Xthreeten$.
In particular, 
we obtain the list of  $ADE$-types $R_{[\phi\sprime]}$ of the reducible fibers 
of genus one fibrations $\phi\sprime$ on
$\Xtwoten$ and $\Xthreeten$.
See Theorems~\ref{thm:fibchar2} and~\ref{thm:fibchar3}.
\par
\medskip
In Elkies and Sch\"utt~\cite{ElkiesSchuett} and Sengupta~\cite{Sengupta}  mentioned above,
they also obtained explicit defining equations of the Jacobian fibrations.
Note that the lattice equivalence classes of all \emph{extremal} (quasi-)~elliptic  fibrations 
(i.e., Jacobian fibrations with Mordell-Weil rank zero)
on supersingular $K3$ surfaces are classified in Shimada~\cite{MR2059747}.
\par
\medskip
As the second application of Lemma~\ref{lem:main}, 
we investigate the automorphism group of the nef cone of a supersingular $K3$ surface.
For a $K3$ surface $Y$,
let $\NE(Y)\subset S_Y\tensor\R$  denote the nef cone.
We denote  by
$ \aut(\NE(Y))\subset \OG^+(S_Y)$
the group of isometries of $S_Y$ that preserve $\NE(Y)$.
Since   $\Aut(\Xps)$ acts on $\Sps$ faithfully (Rudakov and Shafarevich~\cite[Section 8, Proposition 3]{MR633161}),
we have
\begin{equation}\label{eq:auts}
\Aut (\Xps)\subset \aut(\NE(\Xps)) \subset \OG^+(\Sps).
\end{equation}
Using the description of   $\Aut(\Xtwoone)$  
by Dolgachev and Kondo~\cite{MR1935564}, 
and that of $\Aut(\Xthreeone)$ by Kondo and Shimada~\cite{KondoShimada}, 
we give a set of generators of $\aut(\NE(\Xtwoten))$ and $\aut(\NE(\Xthreeten))$
in Theorems~\ref{thm:char2sigma10} and~\ref{thm:char3sigma10}, respectively. 
\par
\medskip
Suppose that $p$ is odd.
We fix a lattice $N$ isomorphic to $\Sps$.
Then a  quadratic space $(N_0, q_0)$ of dimension $2\sigma$
over $\F_p$ is defined by
\begin{equation}\label{eq:N0q0}
N_0:=pN\dual/pN\quand 
q_0(px\bmod pN):= px^2 \bmod p\quad (x\in N\dual).
\end{equation}
We fix a \emph{marking} $\eta: N\isom \Sps$
for  a supersingular $K3$ surface $X:=\Xps$ defined over $k$.
Then $\aut(\NE(X))$ acts on $(N_0, q_0)$,
and  the \emph{period} $K_{(X, \eta)}\subset N_0\tensor k$ of the marked supersingular $K3$ surface 
$(X, \eta)$ is defined as the Frobenius pull-back  of the kernel of the natural homomorphism
$$
N\tensor k \to S_X\tensor k \to H^2_{\DR}(X/k)
$$
(see Section~\ref{sec:Torelli}).
In virtue of   Torelli theorem  for supersingular $K3$ surfaces by  Ogus~\cite{MR563467}, \cite{MR717616},
the subgroup $\Aut(X)$ of $\aut(\NE(X))$
is equal to the stabilizer subgroup of the period $K_{(X, \eta)}$.
In particular,
the index of $\Aut (\Xps)$ in $\aut(\NE(\Xps))$ is finite.
On the other hand,
the classification of $2$-reflective lattices due to Nikulin~\cite{MR633160}
implies that  $\aut(\NE(\Xps))$ is infinite.
Hence, at least when $p$ is odd,
$\Aut (\Xps)$ is an infinite group.
See~Sections~\ref{sec:chamber} and~\ref{sec:Torelli} for details.
Moreover,  Lieblich and Maulik~\cite{11023377} proved that,
if $p>2$, then $\Aut (\Xps)$ is finitely generated and 
its action on $\NE(\Xps)$ has a  rational polyhedral fundamental
domain.
\par
\medskip
We say that a supersingular $K3$ surface $X$ is \emph{generic}
if there exists a marking $\eta: N\isom S_X$ such that
the isometries of $(N_0, q_0)$ that preserve the period $K_{(X, \eta)}\subset N_0\tensor k$
are only scalar multiplications
(see Definition~\ref{def:generic}).
Using the surjectivity of the period mapping proved by Ogus~\cite{MR717616},
we prove the following:
\begin{theorem}\label{thm:existenceofgeneric}
Suppose that $p$ is odd and $\sigma>1$. Then 
 there exist an algebraically closed  field $k$ and a supersingular $K3$ surface $X$ with Artin invariant $\sigma$
defined over $k$ that is  generic.
\end{theorem}
We observe that,
if  $\Xthreeten$ is generic, 
then  the  index of $\Aut(\Xthreeten)$ in $\aut(\NE(\Xthreeten))$ is very large (see~Remark~\ref{rem:heavy}).
An analogous result for a generic complex Enriques surface 
was obtained by Barth and Peters~\cite{MR718937}.

%
\par
\medskip
This paper is organized as follows.
In~Section~\ref{sec:pre},
we fix notation and terminologies about lattices and $K3$ surfaces.
In~Section~\ref{sec:NSssK3},
Lemma~\ref{lem:main} is proved
by means of the fundamental results of Rudakov and Shafarevich~\cite{MR633161}
on the N\'eron-Severi lattices of supersingular $K3$ surfaces.
In~Section~\ref{sec:genusone},
we study genus one fibrations on supersingular $K3$ surfaces,
and prove Theorems~\ref{thm:genusone} and~\ref{thm:nosections1}. 
Moreover,  the bijections $\E(\Xpone)\cong\E(\Xpten)$ for $p=2$ and $3$
are given explicitly 
in Tables~\ref{table:char2} and~\ref{table:char3}.
In~Section~\ref{sec:chamber},
we review the classical method to investigate  
the orthogonal group of a hyperbolic lattice 
by means of a chamber decomposition of the associated hyperbolic space,
and fix some notation and terminologies.
We then apply this method to the nef cone of a supersingular $K3$ surface.
In Section~\ref{sec:autNE23},
we give a set of generators of $\aut(\NE(\Xtwoten))$ and $\aut(\NE(\Xthreeten))$.
In~Section~\ref{sec:Torelli},
we review the theory of  the period mapping and  Torelli theorem for supersingular $K3$ surfaces
in odd characteristics due to Ogus~\cite{MR563467}, \cite{MR717616},
and describe the relation between $\Aut(\Xps)$ and $\aut(\NE(\Xps))$.
In Section~\ref{sec:existenceofgeneric},
we prove Theorem~\ref{thm:existenceofgeneric}.
%
\par
\medskip

{\bf Convention.}
We use $\aut$ to denote  automorphism groups of  lattice theoretic objects,
and $\Aut$ to denote  automorphism groups  of  geometric objects on  $K3$ surfaces.
\section{Preliminaries}\label{sec:pre}
\subsection{Lattices}
A \emph{$\Q$-lattice}  is  a free $\Z$-module  $L$ of  finite rank equipped with 
a non-degenerate symmetric bilinear form $\intM{\cdot , \cdot}{L} : L \times L \to \Q$. 
We omit the subscript $L$ in $\intM{\cdot , \cdot}{L}$
if no confusions will occur.
If $\intM{\cdot , \cdot}{L}$ takes values in $\Z$,
we say that $L$ is a \emph{lattice}.
For $x \in L\otimes \R $, we call $x^2 :=\intM{x,x}{}$ the \emph{norm} of $x$.
A vector in $L\tensor \R$ of norm $n$ is sometimes called an \emph{$n$-vector}.
A lattice $L$ is said to be  \emph{even} if 
$x^2\in 2\Z$ holds
for any $x\in L$. 
\par
Let $L$ be a free $\Z$-module  of  finite rank.
A submodule $M$ of $L$ is \emph{primitive} if $L/M$ is torsion free.
A non-zero vector $v\in L$ is \emph{primitive}
if the submodule of $L$ generated by  $v$ is primitive.
\par
Let $L$ be a $\Q$-lattice of rank $r$.
For  a non-zero rational number
$m$, we denote by $L(m)$ the free $\Z$-module $L$ with 
the symmetric  bilinear form
$\intM{x, y}{L(m)}:=m \intM{x, y}{L}$.
The signature of $L$ is the signature of the real quadratic space $L\tensor\R$. 
We say that $L$ is \emph{negative definite}
if the signature of $L$ is $(0, r)$,
and $L$ is \emph{hyperbolic}
if the signature is $(1, r-1)$.
A \emph{Gram matrix} of $L$ is an $r\times r$ matrix with entries $\intM{e_i, e_j}{}$,
where $\{e_1, \dots, e_r\}$ is a basis of $L$.
The determinant of a Gram matrix of $L$ is called the \emph{discriminant} of $L$.
\par
 For an even lattice $L$,
 the set of $(-2)$-vectors is denoted by $\RRR(L)$.
A \emph{negative} definite even lattice $L$ is called a \emph{root lattice}
if $L$ is generated by $\RRR(L)$.
Let $R$ be an $ADE$-type.
The root lattice of type $R$ is the root lattice whose
Gram matrix is the Cartan matrix of type $R$.
Suppose that $L$ is negative definite.
By the \emph{$ADE$-type of $\RRR(L)$},
we mean the $ADE$-type of the root sublattice $\gen{\RRR(L)}$ of $L$
generated by $\RRR(L)$.
(See, for example, Bourbaki~\cite{MR0240238} for the classification of root lattices.)
\par
Let $L$ be an even lattice and let $L\dual :={\rm Hom}(L,\Z)$ 
be identified with a submodule  of  $L\otimes \Q $ with the extended symmetric bilinear form.
We call this $\Q$-lattice $L\dual$ the \emph{dual lattice} of $L$.
The \emph{discriminant group} of $L$ is defined to be 
the quotient
$L\dual/L$, and is denoted by $A_L$.
We  define the \emph{discriminant quadratic form} of $L$ 
$$q_L : A_L \to \Q /2\Z$$ 
by $q_L(x\bmod L) := x^2 \bmod 2\Z$. 
The order of $A_L$ is equal to the discriminant of $L$ up to sign.
We say that $L$ is \emph{unimodular} if $A_L$ is trivial,
while 
$L$ is \emph{$p$-elementary} if $A_L$  is  $p$-elementary.  
An even  $2$-elementary lattice $L$ is said to be  \emph{of type \typeI} 
if $q_L(x \bmod  L) \in \Z/2\Z$
holds for any $x \in L\dual$.
Note that $L$ is $p$-elementary if and only if $pG_L\inv$ is an integer matrix,
where $G_L$ is a Gram matrix of $L$.
\par
Let ${\OG}(L)$ denote  the orthogonal group of a lattice $L$, that is, 
the group of isomorphisms of $L$ preserving $\intM{\cdot, \cdot}{L}$.
We assume that ${\OG}(L)$ acts on $L$ from \emph{right}, and 
the action of $g\in \OG(L)$ on $v\in L\tensor\R$ is denoted by $v\mapsto v^g$.
Similarly ${\OG}(q_L)$ denotes the group of isomorphisms of $A_L$ preserving $q_L$.
There is a natural homomorphism ${\OG}(L) \to {\OG}(q_L)$.
%
%

Let $L$ be a hyperbolic lattice.
A \emph{positive cone} of $L$ is one of the two connected components of 
$$
\set{x\in L\tensor \R}{x^2>0}.
$$
Let $\PPP_{L}$ be a positive cone of $L$.
We denote by ${\OG}^+(L)$ the group of isometries of $L$ that preserve  $\PPP_{L}$.
We have ${\OG}(L)={\OG}^+(L)\times\{\pm 1\}$.
For a vector $v\in L\tensor \R$ with $v^2<0$,  we put 
$$
(v)\sperp:=\set{x\in \PPP_L}{\intM{x, v}{}=0},
$$
which is   a real hyperplane of $\PPP_L$.
An isometry $g\in \OG^+(L)$  is called a \emph{reflection with respect to $v$}
or a \emph{reflection into  $(v)\sperp$}
if $g$ is of order $2$ and fixes each point of $(v)\sperp$.
An element $r$ of $\RRR(L)$ defines a reflection
$$
s_r : x\mapsto x+\intM{x, r}{} r
$$
with respect to $r$.
We denote by $W(L)$ the subgroup of $\OG^+(L)$ generated by  the set of these reflections 
$\shortset{s_r}{r\in \RRR(L)}$. 
It is obvious that $W(L)$ is normal in $\OG^+(L)$.
\subsection{$K3$ surfaces}
Let $Y$ be a $K3$ surface,
and let $S_Y$ denote the N\'eron-Severi lattice of $Y$.
A smooth rational curve on $Y$ is called a \emph{$(-2)$-curve}.
We denote by $\PPP(Y)\subset S_Y\tensor \R$ the positive cone containing an ample class of $Y$.
Recall that the \emph{nef  cone} $\NE(Y)$ of $Y$ is defined by 
$$
\NE(Y):=\set{x\in S_Y\tensor \R}{\intM{x, [C]}{}\ge 0\;\;\textrm{for any curve $C$ on $Y$}},
$$
where $[C]\in S_Y$ is the class of a curve $C\subset Y$.
Then  $\NE(Y)$ is contained in the closure $\overline{\PPP}(Y)$ of $\PPP(Y)$ in $S_Y\tensor \R$.
We put
$$
\NE\spcirc(Y):=\NE(Y)\cap \PPP(Y)=\set{x\in \NE(Y)}{x^2>0}.
$$
The following is well-known. See, for example, Rudakov and Shafarevich~\cite[Section 3]{MR633161}.
\begin{proposition}\label{prop:nef}
{\rm (1)} We have 
$$
\NE(Y)=\set{x\in S_Y\tensor \R}{\intM{x, [C]}{}\ge 0\;\;\textrm{for any $(-2)$-curve $C$ on $Y$}}.
$$
{\rm (2)} 
If $v\in S_Y$ is contained in $\overline{\PPP}(Y)$,
then there exists $g\in W(S_Y)$ such that $v^g\in \NE(Y)$.
\end{proposition}

\section{N\'eron-Severi lattices of supersingular $K3$ surfaces}\label{sec:NSssK3}
%
Let $X_{p, \sigma}$  be  a supersingular $K3$ surface with Artin invariant $\sigma$
in characteristic $p>0$.
Then the isomorphism class of the N\'eron-Severi lattice $\Sps$ of $X_{p, \sigma}$ depends only on $p$ and $\sigma$, 
and is characterized as follows
(see  Rudakov-Shafarevich~\cite[Sections 3,4 and 5]{MR633161} for the proof). 
\begin{theorem}[\cite{MR633161}] \label{thm:characterizationofSps}
{\rm (1)}
The lattice $\Sps$ is an even hyperbolic $p$-elementary lattice of rank $22$ with discriminant $-p^{2\sigma}$.
Moreover, $S_{2, \sigma}$ is of type $\typeI$.
\par
{\rm (2)}
Suppose that  $N$ is an even hyperbolic $p$-elementary lattice of rank $22$ with discriminant $-p^{2\sigma}$.
When  $p=2$, we further assume that $N$ is of type $\typeI$.
Then $N$ is isomorphic to $\Sps$.
\end{theorem}
Using this theorem,
we can prove Lemma~\ref{lem:main} easily.
\begin{proof}[Proof of Lemma~\ref{lem:main} ]
It is enough to show that $\Sps\dual (p)$ is an even $p$-elementary  lattice
of discriminant $-p^{2\sigma\sprime}$,
and that $S_{2, \sigma}\dual (2)$ is of type $\typeI$.
Since $\Sps$ is $p$-elementary,
we have $p \Sps\dual\subset \Sps$.
Therefore $\Sps\dual (p)$ is a lattice.
Let $G_{p, \sigma}$ be a Gram matrix of  $\Sps$.
Then the determinant of the Gram matrix $p G_{p, \sigma}\inv$ of  $\Sps\dual (p)$ is 
equal to  $p^{22}\cdot \det(G_{p, \sigma})\inv=-p^{2\sigma\sprime}$.
Therefore the discriminant of $\Sps\dual (p)$ is  $-p^{2\sigma\sprime}$.
Since $p(p G_{p, \sigma}\inv)\inv=G_{p, \sigma}$ is an integer matrix,
$\Sps\dual (p)$ is $p$-elementary.
Suppose that $p$ is odd.
Then, for any $\xi \in \Sps\dual$,
we have $p\xi \in \Sps$ and 
hence $\intM{p\xi, p\xi}{\Sps}=p\intM{\xi, \xi}{\Sps\dual (p)}$ is even.
Therefore $\Sps\dual (p)$ is even.
Suppose that $p=2$.
Then, for any $\xi \in S_{2, \sigma}\dual$,
we have $\intM{\xi, \xi}{S_{2, \sigma}\dual}\in\Z$,
because $S_{2, \sigma}$ is of type  $\typeI$.
Therefore $S_{2, \sigma}\dual (2)$ is even.
Moreover, for any $\eta \in (S_{2, \sigma}\dual(2))\dual=S_{2, \sigma}(1/2)$,
we have $\intM{\eta, \eta}{S_{2, \sigma}(1/2)}\in \Z$,
because $S_{2, \sigma}$ is even.
Therefore $S_{2, \sigma}\dual (2)$ is of type $\typeI$.
\end{proof}
\begin{corollary}\label{cor:j}
Suppose that $\sigma+\sigma\sprime=11$.
Then there exists an embedding  of $\Z$-modules 
$j : \Sps\inj \Spss$
that induces an isomorphism of lattices $\Sps\dual(p)\cong \Spss$.
This embedding induces an isomorphism
$$
j_* : \OG(\Sps)\isom \OG(\Spss).
$$
Moreover 
such an embedding $j$ is unique up to
compositions  with elements of $\OG(\Spss)$.
\end{corollary}
\begin{remark}\label{rem:jvinNE}
Suppose that $v\in \Sps$ satisfies $v^2\ge 0$.
Then, by Proposition~\ref{prop:nef}(2),
we can choose $j: \Sps\inj \Spss$ 
in Corollary~\ref{cor:j} in such a way that $j(v)\in \NE(\Xpss)$.
\end{remark}
\section{Genus one fibrations}\label{sec:genusone}
Let $Y$ be a $K3$ surface defined over an algebraically closed field  of arbitrary characteristic.
Recall that $f_\phi\in S_Y$ is the class of a fiber of a genus one fibration $\phi: Y\to \P^1$,
$\E(Y)$ is  the set of lattice equivalence classes 
of genus one fibrations on $Y$, 
and $[\phi]\in \E(Y)$ is the class containing $\phi$.
We  summarize properties of a  genus one fibration $\phi: Y\to \P^1$
that depends only on the class $[\phi]$.
See Sections 3 and  4 of Rudakov and Shafarevich~\cite{MR633161},
and Shioda~\cite{MR1081832} for the proof.

(1)
The fibration $\phi$ admits  a section if and only if there exists 
a $(-2)$-vector $z\in S_Y$ such that $\intM{f_{\phi}, z}{}=1$.

(2)
Note that $f_\phi\in S_Y$ is primitive of norm $0$,
and that $\gen{f_\phi}\sperp/\gen{f_\phi}$ is an even negative definite lattice,
where $\gen{f_\phi}\sperp$ is the orthogonal complement in $ S_Y$ 
of the lattice $\gen{f_\phi}$ of rank $1$ generated by  $f_{\phi}$.
The $ADE$-type of the reducible fibers of $\phi$ is equal to the $ADE$-type of the set  
$\RRR(\gen{f_\phi}\sperp/\gen{f_\phi})$ of $(-2)$-vectors in $\gen{f_\phi}\sperp/\gen{f_\phi}$.

(3)
Suppose that $\phi$ admits a section $Z\subset Y$.
Then $f_{\phi}$ and $[Z]\in S_Y$ 
generate  an even unimodular hyperbolic lattice $U_{\phi}$ of rank $2$ in $S_Y$.
Let $K_{\phi}$ denote the orthogonal complement of $U_{\phi}$  in $S_Y$.
We have an orthogonal direct-sum decomposition
$$
S_Y=U_{\phi}\oplus K_{\phi},
$$
and the lattice $\gen{f_\phi}\sperp/\gen{f_\phi}$ is isomorphic to $K_{\phi}$.
Then the Mordell-Weil group of $\phi$ is isomorphic to $K_{\phi}/\gen{\RRR(K_{\phi})}$,
where $\gen{\RRR(K_{\phi})}$ is the root sublattice of $K_{\phi}$ generated by the $(-2)$-vectors in $K_{\phi}$.

(4)
In characteristic $2$ or $3$, 
$\phi$ is quasi-elliptic  if and only if $\gen{\RRR(K_{\phi})}$ is $p$-elementary
of rank $20$.
\par
\medskip 
As a corollary,
we obtain the following:
\begin{proposition}\label{prop:depends_only_on_latequiv}
Suppose that genus one fibrations $\phi: Y\to \P^1$ and $\psi: Y\to \P^1$ on $Y$
are lattice-equivalent.
Then the following hold.

{\rm (1)}
The fibration $\phi$ admits a section if and only if so does $\psi$.

{\rm (2)}
The $ADE$-type of the reducible fibers of $\phi$ is equal to that of $\psi$.

{\rm (3)}
Suppose that $\phi$ and $\psi$ admit a section.
Then the Mordell-Weil groups for $\phi$ and for $\psi$ are isomorphic.

{\rm (4)}
In characteristic $2$ or $3$,
the fibration $\phi$ is quasi-elliptic if and only if so is $\psi$.
\end{proposition}
\begin{definition}
For a hyperbolic  lattice $S$,
we put 
$$
\widetilde{\EEE}(S):=\shortset{v\in S\tensor \Q  }{v\ne 0, v^2=0}/\Q \sptimes
\quand \EEE(S):=\widetilde{\EEE}(S)/ \OG(S).
$$
\end{definition}
\begin{remark}\label{rem:EEEPPPNE}
Let  a positive cone $\PPP_S$ of $S$ be fixed.
Then each  element of $\widetilde{\EEE}(S)$
is represented by a unique non-zero primitive vector $v\in S$ of norm $0$
that is contained in the closure $\overline{\PPP}_S$ of $\PPP_S$ in $S\tensor\R$.
\end{remark}
In Sections 3 and  4 of Rudakov and Shafarevich~\cite{MR633161},
the following is proved:
\begin{proposition}\label{prop:vfphi}
Let $v$ be a non-zero vector of $S_Y$.
Then there exists a genus one fibration $\phi: Y\to \P^1$
such that $v=f_{\phi}$ if and only if $v$ is primitive,  $v^2=0$, and 
$v\in \NE(Y)$.
\end{proposition}
Combining  Propositions~\ref{prop:nef},~\ref{prop:vfphi} and Remark~\ref{rem:EEEPPPNE},
we obtain the following:
\begin{corollary}
The map $\phi\mapsto f_{\phi}$ induces a bijection from 
$\E(Y)$ to $\EEE(S_Y)$. 
\end{corollary}
%
%
%
From now on, we work over an algebraically closed field of characteristic $p>0$.
\begin{proof}[Proof of Theorem~\ref{thm:genusone}]
Consider the  embedding $j:  \Sps\inj \Spss$ in Corollary~\ref{cor:j}.
Then $j$ 
is unique up to $\OG(\Spss)$,
induces a bijection from $\widetilde{\EEE}(\Sps)$ to $\widetilde{\EEE}(\Spss)$,
and induces an isomorphism $\OG(\Sps)\cong \OG(\Spss)$.
Hence it induces a canonical bijection from $\EEE(\Sps)$ to $\EEE(\Spss)$.
\end{proof}
We denote this canonical one-to-one 
correspondence from $\E(X_{p, \sigma})$ to 
$\E(X_{p, \sigma\sprime})$ by $[\phi]\mapsto [\phi\sprime]$.
\begin{remark}\label{rem:jfphi}
Let a genus one fibration $\phi:\Xps\to \P^1$ be given,
and let $\phi\sprime:\Xpss\to \P^1$ be a representative of 
$[\phi\sprime]$.
Then
we can choose the embedding $j:  \Sps\inj \Spss$ inducing $\Sps\dual(p)\cong \Spss$
in such a way that $j(f_{\phi})$ is a  scalar multiple of $f_{\phi\sprime}$
 by a positive integer.
\end{remark}
\begin{theorem}\label{thm:nosections}
Suppose that a genus one fibration $\phi:\Xps\to\P^1$ admits a section.
Then the corresponding genus one fibration $\phi\sprime :\Xpss\to\P^1$ does not admit a section.
Moreover the $ADE$-type of the reducible fibers of $\phi\sprime$ is equal to
the $ADE$-type of  $\RRR(K_{\phi}\dual (p))$.
\end{theorem}
\begin{proof}
Let $z\in \Sps$ be the class of a section of $\phi$.
We choose $j :  \Sps\inj \Spss$  as in Remark~\ref{rem:jfphi}.
Since $U_\phi\dual=U_\phi$, 
we see that $j(f_{\phi})$ is primitive in $\Spss$
and hence $j(f_{\phi})=f_{\phi\sprime}$.
We have an isomorphism $S_{p, \phi\sprime}\cong U_{\phi}(p)\oplus K_{\phi}\dual(p)$
such that $f_{\phi\sprime}$ and $j(z)$ form a basis of $U_{\phi}(p)$.
Since there are no vectors $v \in U_{\phi}(p)\oplus K_{\phi}\dual(p)$ 
with $\intM{v, f_{\phi\sprime}}{}=1$,
the fibration $\phi\sprime$ does not admit a section.
Moreover the lattice $\gen{f_{\phi\sprime}}\sperp/\gen{f_{\phi\sprime}}$ is isomorphic to $K_{\phi}\dual (p)$.
\end{proof}
%
%
%
The list of lattice equivalence classes of genus one fibrations on
$X_{2,1}$ and $\Xthreeone$ were obtained by Elkies and Sch\"utt~\cite{ElkiesSchuett} and by Sengupta~\cite{Sengupta},
respectively.
From their results, we obtain the following results on supersingular $K3$ surfaces with Artin invariant $10$:
\begin{theorem}\label{thm:fibchar2}
There exist  $18$ lattice equivalence classes of genus one fibrations on  $\Xtwoten$.
The $ADE$-type $R_{[\phi\sprime]}$ of the reducible fibers of each  
$[\phi\sprime]\in \E(\Xtwoten)$ is given at the last column of Table~\ref{table:char2}.
\end{theorem}
\begin{theorem}\label{thm:fibchar3}
There exist   $52$ lattice equivalence classes of genus one fibrations on   $\Xthreeten$.
The $ADE$-type $R_{[\phi\sprime]}$
of the reducible fibers 
of each $[\phi\sprime]\in \E(\Xthreeten)$
 is  given at the last column of Table~\ref{table:char3}.
\end{theorem}
%
%
\begin{table}
$$
\begin{array}{cc|ccc|c}
{\rm No.} & R_N  & &\sigma=1& &\sigma=10 \\
\hline
&  & R_{[\phi]} & {\rm MW}_{\rm tor} & {\rm rank(MW)} &  R_{[\phi\sprime]} \\
\hline
1 & 4A\sb{5} + D\sb{4} & 4A\sb{5} & [3, 6] &0 & 0 \\
2 & 6D\sb{4} & 5D\sb{4} & [2, 2, 2, 2] &0 & 0 \\
3 & 2A\sb{7} + 2D\sb{5} & 2A\sb{7} + D\sb{5} & [8] &1 & A\sb{1} \\
4 & 2A\sb{9} + D\sb{6} & 2A\sb{1} + 2A\sb{9} & [10] &0 & 2A\sb{1} \\
5 & 4D\sb{6} & 2A\sb{1} + 3D\sb{6} & [2, 2, 2] &0 & 2A\sb{1} \\
6 & A\sb{11} + D\sb{7} + E\sb{6} & A\sb{11} + D\sb{7} & [4] &2 & A\sb{2} \\
7 & A\sb{11} + D\sb{7} + E\sb{6} & A\sb{3} + A\sb{11} + E\sb{6} & [6] &0 & 3A\sb{1} \\
8 & 4E\sb{6} & 3E\sb{6} & [3] &2 & A\sb{2} \\
9 & 3D\sb{8} & D\sb{4} + 2D\sb{8} & [2, 2] &0 & 4A\sb{1} \\
10 & A\sb{15} + D\sb{9} & A\sb{15} + D\sb{5} & [4] &0 & 5A\sb{1} \\
11 & A\sb{17} + E\sb{7} & 3A\sb{1} + A\sb{17} & [6] &0 & A\sb{3} \\
12 & D\sb{10} + 2E\sb{7} & 3A\sb{1} + D\sb{10} + E\sb{7} & [2, 2] &0 & A\sb{3} \\
13 & D\sb{10} + 2E\sb{7} & D\sb{6} + 2E\sb{7} & [2] &0 & 6A\sb{1} \\
14 & 2D\sb{12} & D\sb{8} + D\sb{12} & [2] &0 & 8A\sb{1} \\
15 & D\sb{16} + E\sb{8} & D\sb{4} + D\sb{16} & [2] &0 & D\sb{4} \\
16 & D\sb{16} + E\sb{8} & D\sb{12} + E\sb{8} & [1] &0 & 12A\sb{1} \\
17 & 3E\sb{8} & D\sb{4} + 2E\sb{8} & [1] &0 & D\sb{4} \\
18 & D\sb{24} & D\sb{20} & [1] &0 & 20A\sb{1} \\
\end{array}
$$
\caption{Genus one fibrations on  $\Xtwoone$ and $\Xtwoten$}\label{table:char2}
\end{table}
\begin{table}
{\small 
$$
\begin{array}{cc|ccc|c}
{\rm No.} & R_N  & &{\sigma=1}& &\sigma=10 \\
\hline
&  & R_{[\phi]} & {\rm MW}_{\rm tor} & {\rm rank(MW)} &  R_{[\phi\sprime]} \\
\hline
1 & 12A\sb{2} & 10A\sb{2} & [3, 3, 3, 3] &0 & 0 \\
2 & 8A\sb{3} & 6A\sb{3} & [4, 4] &2 & 0 \\
3 & 6A\sb{4} & 2A\sb{1} + 4A\sb{4} & [5] &2 & 0 \\
4 & 6D\sb{4} & 4D\sb{4} & [2, 2] &4 & 0 \\
5 & 4A\sb{5} + D\sb{4} & A\sb{2} + 3A\sb{5} & [3] &3 & 0 \\
6 & 4A\sb{5} + D\sb{4} & 3A\sb{5} + D\sb{4} & [2, 6] &1 & A\sb{1} \\
7 & 4A\sb{5} + D\sb{4} & 2A\sb{2} + 2A\sb{5} + D\sb{4} & [2] &2 & 0 \\
8 & 4A\sb{6} & 3A\sb{6} & [7] &2 & A\sb{1} \\
9 & 4A\sb{6} & 2A\sb{3} + 2A\sb{6} & [1] &2 & 0 \\
10 & 2A\sb{7} + 2D\sb{5} & 4A\sb{1} + 2A\sb{7} & [2, 4] &2 & 0 \\
11 & 2A\sb{7} + 2D\sb{5} & A\sb{1} + A\sb{7} + 2D\sb{5} & [4] &2 & A\sb{1} \\
12 & 2A\sb{7} + 2D\sb{5} & 2A\sb{1} + A\sb{4} + A\sb{7} + D\sb{5} & [2] &2 & 0 \\
13 & 2A\sb{7} + 2D\sb{5} & 2A\sb{4} + 2D\sb{5} & [1] &2 & 0 \\
14 & 3A\sb{8} & A\sb{2} + 2A\sb{8} & [3] &2 & A\sb{1} \\
15 & 3A\sb{8} & 2A\sb{5} + A\sb{8} & [1] &2 & 0 \\
16 & 4D\sb{6} & 3D\sb{6} & [2, 2] &2 & 2A\sb{1} \\
17 & 4D\sb{6} & 2A\sb{3} + 2D\sb{6} & [2, 2] &2 & 0 \\
18 & 2A\sb{9} + D\sb{6} & 2A\sb{9} & [5] &2 & 2A\sb{1} \\
19 & 2A\sb{9} + D\sb{6} & A\sb{3} + A\sb{9} + D\sb{6} & [2] &2 & A\sb{1} \\
20 & 2A\sb{9} + D\sb{6} & A\sb{3} + A\sb{6} + A\sb{9} & [1] &2 & 0 \\
21 & 2A\sb{9} + D\sb{6} & 2A\sb{6} + D\sb{6} & [1] &2 & 0 \\
22 & 4E\sb{6} & A\sb{2} + 3E\sb{6} & [3] &0 & A\sb{2} \\
23 & 4E\sb{6} & 4A\sb{2} + 2E\sb{6} & [3, 3] &0 & 0 \\
24 & A\sb{11} + D\sb{7} + E\sb{6} & A\sb{2} + A\sb{11} + D\sb{7} & [4] &0 & A\sb{2} \\
25 & A\sb{11} + D\sb{7} + E\sb{6} & A\sb{11} + E\sb{6} & [3] &3 & 2A\sb{1} \\
26 & A\sb{11} + D\sb{7} + E\sb{6} & 2A\sb{2} + A\sb{11} + D\sb{4} & [6] 
&1 & 0 \\
27 & A\sb{11} + D\sb{7} + E\sb{6} & A\sb{5} + D\sb{7} + E\sb{6} & [1] &2 & A\sb{1} \\
28 & A\sb{11} + D\sb{7} + E\sb{6} & 2A\sb{2} + A\sb{8} + D\sb{7} & [1] &1 & 0 \\
29 & A\sb{11} + D\sb{7} + E\sb{6} & A\sb{8} + D\sb{4} + E\sb{6} & [1] &2 & 0 \\
30 & 2A\sb{12} & A\sb{6} + A\sb{12} & [1] &2 & A\sb{1} \\
31 & 2A\sb{12} & 2A\sb{9} & [1] &2 & 0 \\
32 & 3D\sb{8} & 2A\sb{1} + 2D\sb{8} & [2, 2] &2 & 2A\sb{1} \\
33 & 3D\sb{8} & 2D\sb{5} + D\sb{8} & [2] &2 & 0 \\
34 & A\sb{15} + D\sb{9} & A\sb{3} + A\sb{15} & [4] &2 & 2A\sb{1} \\
35 & A\sb{15} + D\sb{9} & A\sb{9} + D\sb{9} & [1] &2 & A\sb{1} \\
36 & A\sb{15} + D\sb{9} & A\sb{12} + D\sb{6} & [1] &2 & 0 \\
37 & A\sb{17} + E\sb{7} & A\sb{2} + A\sb{17} & [3] &1 & A\sb{1} + A\sb{2} \\
38 & A\sb{17} + E\sb{7} & A\sb{11} + E\sb{7} & [1] &2 & A\sb{1} \\
39 & A\sb{17} + E\sb{7} & A\sb{5} + A\sb{14} & [1] &1 & 0 \\
40 & D\sb{10} + 2E\sb{7} & A\sb{2} + D\sb{10} + E\sb{7} & [2] &1 & A\sb{1} + A\sb{2} \\
41 & D\sb{10} + 2E\sb{7} & 2A\sb{5} + D\sb{10} & [2, 2] &0 & 0 \\
42 & D\sb{10} + 2E\sb{7} & D\sb{4} + 2E\sb{7} & [2] &2 & 2A\sb{1} \\
43 & D\sb{10} + 2E\sb{7} & A\sb{5} + D\sb{7} + E\sb{7} & [2] &1 & 0 \\
44 & 2D\sb{12} & D\sb{6} + D\sb{12} & [2] &2 & 2A\sb{1} \\
45 & 2D\sb{12} & 2D\sb{9} & [1] &2 & 0 \\
46 & 3E\sb{8} & 2A\sb{2} + 2E\sb{8} & [1] &0 & 2A\sb{2} \\
47 & 3E\sb{8} & 2E\sb{6} + E\sb{8} & [1] &0 & 0 \\
48 & D\sb{16} + E\sb{8} & 2A\sb{2} + D\sb{16} & [2] &0 & 2A\sb{2} \\
49 & D\sb{16} + E\sb{8} & D\sb{10} + E\sb{8} & [1] &2 & 2A\sb{1} \\
50 & D\sb{16} + E\sb{8} & D\sb{13} + E\sb{6} & [1] &1 & 0 \\
51 & A\sb{24} & A\sb{18} & [1] &2 & A\sb{1} \\
52 & D\sb{24} & D\sb{18} & [1] &2 & 2A\sb{1} \\
\end{array}
$$
}
\caption{Genus one fibrations on $\Xthreeone$ and $\Xthreeten$}\label{table:char3}
\end{table}
In Table~\ref{table:char2} (resp.~Table~\ref{table:char3}),
the lists  $\E(\Xtwoone)$ and $\E(\Xtwoten)$
(resp.~$\E(\Xthreeone)$ and  $\E(\Xthreeten)$) are presented.
Two lattice equivalence classes in the same row are the pair of $[\phi]\in \E(\Xpone)$
and its corresponding partner $[\phi\sprime]\in \E(\Xpten)$.
The $ADE$-type $R_{[\phi]}$ of the reducible fibers of $\phi$,
and 
the torsion ${\rm MW}_{\rm tor}$ and the rank  of the Mordell-Weil group of $\phi$
are also given.
(Recall that $\phi$ is Jacobian for any $[\phi]\in \E(\Xpone)$  by Elkies and Sch\"utt~\cite{ElkiesSchuett}.)
The meaning of the entry $R_N$ is explained in the proof of Theorems~\ref{thm:fibchar2} and~\ref{thm:fibchar3}.
\begin{proof}[Proof of Theorems~\ref{thm:fibchar2} and~\ref{thm:fibchar3}]
By Theorem~\ref{thm:nosections}, 
it is enough to calculate the $ADE$-type of $\RRR(K_{\phi}\dual (p))$
for $p=2, 3$ and $[\phi]\in \E(\Xpone)$.
The lattices $K_{\phi}$ are calculated in 
Elkies and Sch\"utt~\cite{ElkiesSchuett} and Sengupta~\cite{Sengupta}.
We put
$$
T:=\textrm{the  root lattice of type}\begin{cases}
D_4 & \textrm{if $p=2$,}\\
2A_2 & \textrm{if $p=3$}.
\end{cases}
$$
Then, for each $[\phi]\in \E(\Xpone)$,
there exist a Niemeier lattice $N_{\phi}$ and a primitive  embedding of $T$ into $N_{\phi}$
such that $K_{\phi}$ is  isomorphic to the orthogonal complement of $T$ in $N_{\phi}$.
The entry $R_N$ in Tables~\ref{table:char2} and~\ref{table:char3} 
indicates the $ADE$-type of  $\RRR(N_{\phi})$.
From a Gram matrix of $K_{\phi}$, 
we can calculate 
the $ADE$-type of 
$\RRR(K_{\phi}\dual(p))$
by the algorithm described in~\cite[Section 4]{MR2036331} or~\cite[Section 3]{shimadapreprintchar5}.
\end{proof}
\begin{corollary}
There exist no quasi-elliptic fibrations on $\Xthreeten$.
\end{corollary}
\begin{remark}
Rudakov and Shafarevich~\cite[Section 5]{MR633161} showed  that 
there  exists a quasi-elliptic fibration on $X_{2, \sigma}$
for any $\sigma$.
The  quasi-elliptic fibration on $\Xtwoten$ (No. 18 of Table~\ref{table:char2}) 
was discovered  by Rudakov and Shafarevich~\cite[Section 4]{MR508830}.
\end{remark}
\section{Chamber decomposition of a positive cone}\label{sec:chamber}
Let $S$ be an even hyperbolic lattice,
and let $\PPP_S\subset S\tensor \R$ be a positive cone.
In this section,
we review a general method to find a set of generators of a subgroup of $\OG^+(S)$
by means of a chamber decomposition of $\PPP_S$,
which was developed by Vinberg~\cite{MR0422505}, 
Conway~\cite{MR690711} and Borcherds~\cite{MR913200}. 
%
\par
\medskip
%
Any real hyperplane in $\PPP_S$  is 
written in  the form $(v)\sperp$ by some vector $v\in S\tensor\R$
with negative norm.
We denote by $\HHH_S$ the set of  real hyperplanes in $\PPP_S$, which is canonically identified with 
%
%
$$
\set{v\in S\tensor \R}{v^2<0}/\R\sptimes.
$$
%
%
For a subset $V$ of  $\shortset{v\in S\tensor \R}{v^2<0}$,
we denote by $V\sphyp\subset \HHH_S$ the image of $V$ by $v\mapsto (v)\sperp$.
A closed subset $D$ of $\PPP_S$ is called a \emph{chamber} if
the interior $D\spcirc$  of  $D$ is non-empty and  
there exists 
a set $\Delta_D$ of vectors $v\in S\tensor \R$ with $v^2<0$ such that
$$
D=\set{x\in \PPP_S}{\intM{x, v}{}\ge 0\;\;\textrm{for all}\;\; v\in \Delta_D}.
$$
A hyperplane $(v)\sperp$ of $\PPP_S$ is called a \emph{wall} of $D$ if  
$D\spcirc\cap (v)\sperp=\emptyset$ and 
$D\cap (v)\sperp$ contains an open subset of $(v)\sperp$.
When $D$ is a chamber, we always assume that the set $\Delta_D$ is minimal in the sense
that, for any $v\in\Delta_D$, there exists a point $x\in \PPP_S$ such that $\intM{x, v}{}<0$ and
$\intM{x, v\sprime}{}\ge 0$ for any $v\sprime \in \Delta_D\setminus \{v\}$,
that is, the projection $\Delta_D\to \Delta_D^*$ is bijective and 
every  hyperplane $(v)\sperp\in \Delta_D\sphyp$ is a wall of  $D$.
\par
\medskip
For a chamber $D$, we put
$$
\aut(D):=\set{g\in \OG^+(S)}{D^g=D}.
$$
A chamber $D$ is said to be \emph{fundamental} if the following hold:
\begin{itemize}
\item[(i)]  $\PPP_S$ is the union of all $D^g$, where $g$ runs through $\OG^+(S)$, and 
\item[(ii)]  if $D\spcirc \cap D^g\ne \emptyset$, then $g\in \aut(D)$.
\end{itemize}
Let $\FFF$ be a family of hyperplanes  in $\PPP_S$ with the following properties:
\begin{itemize}
\item[(a)]  $\FFF$  is locally finite in $\PPP_S$ , and 
\item[(b)]  $\FFF$ is invariant under the action of  $\OG^+(S)$ on $\HHH_S$.
\end{itemize}
Then the closure  of each connected component of
$$
\PPP_S\setminus \bigcup_{\FFF}  \;(v)\sperp
$$
is a chamber, 
which we call an \emph{$\FFF$-chamber}.
\par
\medskip
Suppose that $D$ is an $\FFF$-chamber.
Then
$D^g$ is also an $\FFF$-chamber for any $g\in \OG^+(S)$ by the property (b) of $\FFF$,  and
$D$ satisfies the property (ii) in the definition of fundamental   chambers.
Moreover, $D$ satisfies the property (i) if and only if 
every $\FFF$-chamber is equal to  $D^g$ for some $g\in \OG^+(S)$.

For each wall $(v)\sperp\in \Delta_D\sphyp$ of an $\FFF$-chamber $D$, there exists a unique $\FFF$-chamber $D\sprime$
distinct from $D$
such that $D\cap D\sprime\cap (v)\sperp$ contains an open subset of $(v)\sperp$.
We say that $D\sprime$ is \emph{adjacent to $D$ along $(v)\sperp$},
and that $(v)\sperp$ is the  \emph{wall between the adjacent chambers $D$ and $D\sprime$}.
\begin{proposition}
An $\FFF$-chamber $D$ is fundamental
if and only if, for each $v\in \Delta_{D}$, 
there exists $g_v\in \OG^+(S)$ such that $D^{g_v}$ is adjacent to $D$ along $(v)\sperp$.
\end{proposition}
\begin{proof}
The `only if\,' part is obvious.
We prove the `if\,' part.
It is enough to show that, for 
an arbitrary $\FFF$-chamber $D\sprime$,
there exists $g\in \OG^+(S)$ such that $D\sprime=D^g$.
Since the family $\FFF$ of hyperplanes is locally finite in $\PPP_S$,
there exists a finite chain of $\FFF$-chambers $D_0=D, D_1, \dots, D_N=D\sprime$
such that $D_{i}$ and $D_{i+1}$ are adjacent.
We show, by induction on $N$,  that there exists a sequence 
of vectors $v_1, \dots, v_N$ in $\Delta_D$
such that $D_i=D^{g_{v_i} \cdots g_{v_1}}$ holds for $i=1, \dots, N$.
The case $N=0$ is trivial.
Suppose that $N>0$.
Let $(w)\sperp$ be the wall between $D_{N-1}$ and $D_N$,
and let $v_N\in \Delta_D$ be the vector such  that
the wall $(v_N)\sperp$ of $D$ is mapped to the wall $(w)\sperp$ of $D_{N-1}$ by
$g_{v_{N-1}} \cdots g_{v_{1}}$.
Then we have 
$D_N=D^{g_{v_N} \cdots g_{v_1}}$.
\end{proof} 
\begin{remark}
If an $\FFF$-chamber  is fundamental, then any $\FFF$-chamber  is fundamental.
\end{remark}
Let $\GGG$ be a subset of $\FFF$ that is invariant under the action of $\OG^+(S)$.
Then $\GGG$ is locally finite, and any $\GGG$-chamber is a union of $\FFF$-chambers.
If an $\FFF$-chamber is fundamental,
then any $\GGG$-chamber is also fundamental.
%
\begin{proposition}\label{prop:FFFGGG}
Let $D$ be an $\FFF$-chamber and  
let $C$ be a $\GGG$-chamber  such that
$D\subset C$.
Suppose that  $D$ is fundamental.
For $v\in \Delta_D$, 
let $g_v\in \OG^+(S)$ be an isometry such that $D^{g_v}$ is adjacent to $D$ along $(v)\sperp$.
We put 
$$
\Gamma:=\set{g_v}{v\in \Delta_D, (v)\sperp\notin \GGG}.
$$
Then $\aut(C)$
is generated by $\aut(D)$ and $\Gamma$.
\end{proposition}
\begin{proof}
If $g_v\in \Gamma$, then $D^{g_v}$ is contained in $C$
because the wall $(v)\sperp$ between $D$ and $D^{g_v}$ does not belong to  $\GGG$,
and hence $g_v\in \aut(C)$.
Therefore 
the subgroup $\gen{\aut(D), \Gamma}$ of $\OG^+(S)$ generated by $\aut(D)$ and $\Gamma$ is contained in 
$\aut(C)$.
To prove $\aut(C)\subset \gen{\aut(D), \Gamma}$,
it is enough to show that,
for any $g\in \aut(C)$, there exists a sequence 
$\gamma_1, \dots, \gamma_N$ of elements of $\Gamma$ such that 
$D^{g}=D^{\gamma_N\cdots \gamma_1}$.
There exists a sequence of $\FFF$-chambers $D_0=D, D_1, \dots, D_N=D^g$
such that each $D_i$ is contained in $C$ and that $D_{i+1}$ is adjacent to $D_i$ 
for $i=0, \dots, N-1$.
Suppose that we have constructed $\gamma_1, \dots, \gamma_i \in \Gamma$
such that $D_i=D^{\gamma_i\dots \gamma_1}$ holds.
The wall $(w)\sperp$  between $D_{i}$ and  $D_{i+1}$ does not belong to $\GGG$.
Let $v_{i+1}$ be an element of $\Delta_D$ such that the wall
$(v_{i+1})\sperp$ of $D$  is mapped to 
the wall $(w)\sperp$ of $D_i$ by $\gamma_i\dots \gamma_1$.
Since $\GGG$ is  invariant under the action of $\OG^+(S)$,
we have $(v_{i+1})\sperp\notin \GGG$ and hence
$\gamma_{i+1}:=g_{v_{i+1}}$
is an element of  $\Gamma$. Then 
$D_{i+1}=D^{\gamma_{i+1}\gamma_i\cdots \gamma_1}$ holds.
\end{proof}
\begin{remark}
Let $D$ and $C$ be as in Proposition~\ref{prop:FFFGGG}.
Let $v$ and $v\sprime$ be elements of $\Delta_D$.
Suppose that the wall $(v)\sperp$ of $D$ is mapped to the wall $(v\sprime)\sperp$  of $D$
by $h\in \aut(D)$.
Then $D^{h g_{v\sprime} h\inv}$
is adjacent to $D$ along $(v)\sperp$ .
Let $\Delta\sprime_D$ be a subset of $\Delta_D$ such that
 the subset $\Delta\sp{\prime *}_D$ of $\Delta\sphyp_D$ is a complete set of representatives of 
 the orbit decomposition of $\Delta\sphyp_D$ by the action of $\aut(D)$.
 Then  $\aut(C)$
is generated by $\aut(D)$ and $\shortset{g_v}{v\in \Delta\sprime_D, (v)\sperp\notin \GGG}$.
\end{remark}
Considering the case  $\GGG=\emptyset$,
we obtain the following:
\begin{corollary}
Let $D$ be an $\FFF$-chamber.
If $D$ is fundamental,
then $\OG^+(S)$ is generated by $\aut(D)$ and the isometries $g_v$
that map $D$ to its adjacent chambers.
\end{corollary}
\begin{example}\label{example:ref}
Recall that  
$W(S)\subset \OG^+(S)$ is the subgroup generated by $\shortset{s_r}{r\in \RRR(S)}$.
Any $\RRR(S)\sphyp$-chamber  is fundamental,
because  
every $r\in \RRR(S)$ defines a reflection $s_r$.
It follows that  $\OG^+(S)$ is equal  to the semi-direct product of $W(S)$ and 
the automorphism group $\aut(D)$ of an $\RRR(S)\sphyp$-chamber $D$.
In particular,
we have 
$$
\aut(D)\cong \OG^+(S)/W(S).
$$
\end{example}
Let $L$ be an even \emph{unimodular} hyperbolic lattice,
and  let $\iota: S \inj L$ be a primitive embedding.
Let $\PPP_L$ be the positive cone of $L$  that contains $\iota(\PPP_S)$.
We denote by $\Tiota$ the orthogonal complement of $S$ in $L$,
and by
$$
v\mapsto v_S
$$
the orthogonal projection $L\tensor\R\to S\tensor\R$ .
Since  $L$ is a submodule of $S\dual\oplus \Tiota\dual$,
the image of $L$ by $v\mapsto v_S$ is contained in $S\dual$.
We assume the following:
\begin{equation}\label{eq:hypothesisT}\textrm{the natural homomorphism $\OG(\Tiota)\to \OG(q_{\Tiota})$ is surjective.}
\end{equation}
Then we have the following:
\begin{proposition}\label{prop:lift}
For any $g\in \OG^+(S)$, 
there exists $\tilde{g}\in \OG^+(L)$
such that $\iota(v^g)=\iota(v)^{\tilde{g}}$ holds for any $v\in S\tensor \R$. 
\end{proposition}
\begin{proof}
See Nikulin~\cite[Proposition 1.6.1]{MR525944}. 
\end{proof}
%
A hyperplane $(r)\sperp$ of $\PPP_L$ defined by a $(-2)$-vector 
 $r\in \RRR(L)$ intersects $\iota(\PPP_S)$
if and only if $r_S^2<0$.
We put
$$
\RRR(L,  \iota):=\set{r_S}{r\in \RRR(L)\;\;\textrm{and}\;\; r_S^2<0}\;\;\subset\;\; S\dual.
$$
Since $\Tiota$ is negative definite,
we have $-2\le r_S^2$ for any $r\in \RRR(L)$.
Since $S\dual$ is discrete in $S\tensor \R$,
the family of hyperplanes $\RRR(L,  \iota)\sphyp$ is locally finite in $\PPP_S$.
By Proposition~\ref{prop:lift},
if $r\in \RRR(L)$ satisfies $r_S\in \RRR(L,  \iota)$,
then, for any $g\in \OG^+(S)$,
we have $r_S^g=(r^{\tilde{g}})_S \in \RRR(L,  \iota)$.
Therefore $\RRR(L,  \iota)$ is invariant under the action of $\OG^+(S)$.
Note that $\RRR(L,  \iota)$ contains $\RRR(S)$,
and that $\RRR(S)$ is obviously  invariant under the action of $\OG^+(S)$.
Therefore, by Proposition~\ref{prop:FFFGGG},
we can obtain a set of generators of the automorphism group
$\aut(C)$ of an  $\RRR(S)\sphyp$-chamber $C$ 
if we 
find an $\RRR(L,  \iota)\sphyp$-chamber $D$ contained in $C$, 
show that $D$ is fundamental, 
calculate the group $\aut(D)$,
and find   isometries of $S$ that map $D$
to its adjacent chambers.
%
\par
\medskip
Let $\Lts$ denote an even hyperbolic unimodular lattice of rank $26$,
which is unique up to isomorphisms.
The walls of  an $\RRR(\Lts)\sphyp$-chamber $\DDD\subset \Lts\tensor\R$
and the group $\aut(\DDD)\subset \OG^+(\Lts)$ were determined by Conway~\cite{MR690711}.
Then Borcherds~\cite{MR913200}
determined the structure  of $\OG^+(S)$  
for  some even hyperbolic lattices $S$ of rank $<26$ 
by  embedding  $S$ into $\Lts$ in such a way that 
$\Tiota$ is a  root lattice.
\par
\medskip
Kondo~\cite{MR1618132} first applied the Borcherds method to the study of the automorphism group 
of a generic Jacobian Kummer surface.
Later
Dolgachev and Kondo~\cite{MR1935564} applied the Conway-Borcherds method to the study of the automorphism group of $\Xtwoone$,
and 
Kondo and Shimada~\cite{KondoShimada} applied it to $\Xthreeone$.
\par
\medskip
%
%
We say that an even hyperbolic lattice $S$ is \emph{$2$-reflective} if the index of $W(S)$ in $\OG^+(S)$ is finite,
or equivalently,
if the automorphism group of an $\RRR(S)\sphyp$-chamber  is finite (see Example~\ref{example:ref}).
%
Nikulin~\cite{MR633160} classified all $2$-reflective lattices  of rank $\ge 5$.
It turns out that there are no $2$-reflective lattices of rank $>19$.
%
%
%
\par
\medskip
Let $Y$ be a $K3$ surface
with the N\'eron-Severi lattice $S_Y$
and the positive cone $\PPP(Y)$ containing an ample class.
Then
the closed subset  $\NE\spcirc(Y)=\NE(Y)\cap \PPP(Y)$  of $\PPP(Y)$ is an $\RRR(S_Y)\sphyp$-chamber
by Proposition~\ref{prop:nef}(1),
and hence we have 
$$
\aut(\NE(Y))=\aut(\NE\spcirc(Y))\cong \OG^+(S_Y)/W(S_Y).
$$
Therefore 
$\aut(\NE(\Xps))$ is  infinite for any supersingular $K3$ surface $\Xps$.
\erase{
Combining this fact with Nikulin's classification of $2$-reflective lattices, we obtain the following:
\begin{corollary}\label{cor:infiniteautNE}
For any supersingular $K3$ surface $\Xps$,
the group $\aut(\NE(\Xps))$ is  infinite.
\end{corollary}
}
\section{The groups $\aut(\NE(\Xtwoten))$ and $\aut(\NE(\Xthreeten))$}\label{sec:autNE23}
\subsection{The group $\aut(\NE(\Xtwoten))$}\label{subsec:AutNEchar2}
By Lemma~\ref{lem:main}, the result of Dolgachev and Kondo~\cite{MR1935564},
and the method of the previous section, 
we obtain a set of generators of $\aut(\NE(\Xtwoten))$.
First we recall the results of~\cite{MR1935564}.
As a projective model of $\Xtwoone$,  
we consider   the minimal resolution $X$ of the inseparable double cover
$Y\to \P^2$  of $\P^2$
defined by 
$$
w^2=x_0x_1x_2(x_0^3 + x_1^3 +x_2^3).
$$
Note that the projective plane $\P^2(\F_4)$ defined over $\F_4$
contains $21$ points $p_1, \dots , p_{21}$ and $21$ lines $\ell_1, \dots , \ell_{21}$.  
The  inseparable double cover $Y$ has $21$ ordinary nodes
over the $21$ points in $\P^2(\F_4)$ and hence $X$ has $21$ disjoint  $(-2)$-curves.
We denote by 
$e_1, \dots , e_{21}\in \Stwoone$ the classes of these $(-2)$-curves,
by $h\in  \Stwoone$ the class of 
the pullback of a line on $\P^2$,
 and  by $f_1, \dots , f_{21}\in \Stwoone$ 
the classes of the proper transforms of the $21$ lines in $\P^2(\F_4)$.  
Then $\Stwoone$ is generated by the $(-2)$-vectors $e_1, \dots , e_{21}, f_1, \dots , f_{21}$.
The vector
$$
w_M := \frac{1}{3}\sum_{i=1}^{21} (e_i + f_i)
$$
has the property
$$
w_M\in S_X, \,\,  w_M^2 = 14, \,\, \intM{w_M, e_i}{}  =\intM{w_M, f_i}{}  =1.
$$
The complete linear system associated with the line bundle corresponding to $w_M$
defines an embedding of $X$ into $\P^2\times \P^2$, and its image $X_M\subset \P^2\times \P^2$ is defined by
$$
\begin{cases}
x_0 y_0^2+x_1 y_1^2+x_2 y_2^2\;\;=\;\;0, \\
x_0^2 y_0+x_1^2 y_1+x_2^2 y_2\;\;=\;\;0.
\end{cases}
$$
Six points on $\P^2(\F_4)$ are said to be  \emph{general} if no three points of them are
collinear.  There exist $168$ sets of general six points in $\P^2(\F_4)$.
If $I=\{p_{i_1}, \dots, p_{i_6}\}$ is a set of general six points, 
then the $(-1)$-vector
$$
c_{I}:=h - {1\over 2}(e_{i_1} + \cdots + e_{i_6})
$$
is contained in $\Stwoone\dual$.  
Note that 
each $c_{I}$ defines a reflection 
$$
 x\mapsto x+2\intM{x, c_I}{}c_I
$$
in $\OG^+(\Stwoone)$ because
$c_{I} \in  \Stwoone\dual$.
Let $\PPP(\Xtwoone)$ be the positive cone of $\Stwoone$ containing an ample class.
and 
let
$\Delta(\Xtwoone)$ be the set consisting of  $e_1, \dots , e_{21}, f_1, \dots , f_{21}$ and
the $(-1)$-vectors $c_{I}$ defined above. 
We define a chamber $D(\Xtwoone)$ in $P(\Xtwoone)$ by 
\begin{equation*}\label{fdomain}
D(\Xtwoone) := \set{x \in P(\Xtwoone)}{\intM{x, v}{}\ge0\;\;\textrm{for all}\;\; v\in \Delta(\Xtwoone)}.
\end{equation*}
Then, for each $v\in \Delta(\Xtwoone)$, the hyperplane $(v)\sperp$ is a wall of $D(\Xtwoone)$.
Moreover the ample class $w_M$ is contained in the interior of $D(\Xtwoone) $.
Recall that 
 $\Lts$ is the even unimodular hyperbolic lattice of rank $26$.
There exists a primitive embedding $\iota : \Stwoone\inj \Lts$,
which is  unique up to $\OG(\Lts)$.
The orthogonal complement 
$\Tiota$ of $\Stwoone$ in $\Lts$ is isomorphic to the root lattice of type $D_4$, and hence
 satisfies the hypothesis~\eqref{eq:hypothesisT}.
\begin{proposition}
The chamber $D(\Xtwoone)$ is an $\RRR(\Lts, \iota)\sphyp$-chamber
contained in the $\RRR(\Stwoone)\sphyp$-chamber $\NE\spcirc(\Xtwoone)$.
An isometry $g\in \OG^+(\Stwoone)$ belongs to  $\aut(D(\Xtwoone))$ if and only if
$w_M^g=w_M$.
\end{proposition}
Thus we can apply  Proposition~\ref{prop:FFFGGG} 
to  the pair of chambers $D(\Xtwoone)$ and $\NE\spcirc(\Xtwoone)$
for the study of $\aut(\NE(X_{2,1}))$ and $\Aut(\Xtwoone)$.
\par
\medskip
We have the following elements  in $\Aut(\Xtwoone)$ and $\OG^+(\Stwoone)$.
Since $\Aut(\Xtwoone)$ is naturally embedded in $\OG^+(\Stwoone)$,
we use the same letter to denote an element of $\Aut(\Xtwoone)$ and its image in $\OG^+(\Stwoone)$.
\begin{itemize}
\item
The action of $\PGL(3, \F_4)$ on $\P^2$ induces  automorphisms of
the inseparable double cover $Y$ of $\P^2$,
and hence automorphisms of $\Xtwoone$.
Their action on $\Stwoone$ preserves $D(\Xtwoone)$.
\item
The interchange of the two factors of $\P^2\times\P^2$ preserves $X_M\subset \P^2\times\P^2$,
and hence it induces an involution  $\switch\in \Aut(\Xtwoone)$,
which we call the \emph{switch}.
Its action on $\Stwoone$ preserves $D(\Xtwoone)$.
\item
For each set $I$ of general six points in $\P^2(\F_4)$, 
the linear system of plane curves of degree $5$ that pass through the points of $I$
and are singular at each point of $I$
defines a birational involution  of $\P^2$,
and this involution  lifts to an involution of $Y$. 
Hence we obtain an involution $\Cremona_I\in \Aut(\Xtwoone)$,
which we call a \emph{Cremona automorphism} of $\Xtwoone$.
The action of $\Cremona_I$ on $\Stwoone$ is the reflection with respect to $c_I\in \Stwoone\dual$.
\item
The Frobenius action of $\Gal(\F_{4}/\F_2)$ on $X_M$ induces
an isometry $\Frob$ of $\Stwoone$, which preserves $D(\Xtwoone)$.
\item
The $(-2)$-vectors $e_i$ and $f_i$ defines the reflections $s_{e_i}$ and $s_{f_i}$.
\erase{
We have 
the reflections $s_{e_i}$ and $s_{f_i}$ with respect to the $(-2)$-vectors $e_i$ and $f_i$}
\end{itemize}
By the reflections $\Cremona_I, s_{e_i}$ and $ s_{f_i}$,
we see that the chamber  $D(\Xtwoone)$ is fundamental.
\begin{theorem}[\cite{MR1935564}] \label{thm:char2sigma1}
{\rm (1)}
The projective automorphism group $\Aut(\Xtwoone, w_M)$ of 
$X_M\subset \P^2\times \P^2$
is generated by $\PGL(3, \F_4)$ and the switch $\switch$.
\par
{\rm (2)}
The group $\aut(D(\Xtwoone))$ 
is generated by $\Aut(\Xtwoone, w_M)$ and $\Frob$.
\par
{\rm (3)}
The automorphism group $\Aut(\Xtwoone)$ is generated by
$\Aut(\Xtwoone, w_M)$ and the $168$ Cremona automorphisms $\Cremona_I$.
\par
{\rm (4)}
The  group $\aut(\NE(\Xtwoone))$ 
is generated by
$\Aut(\Xtwoone)$ and $\Frob$.
\par
{\rm (5)}
The group $\OG^+(\Stwoone)$ is generated by
$\aut(\NE(\Xtwoone))$ and the $21+21$ reflections $s_{e_i}$ and $s_{f_i}$.
\end{theorem}
We then study $\aut(\NE(\Xtwoten))$.
By Corollary~\ref{cor:j},
we have an embedding 
$$
j : S_{2,1}\inj \Stwoten
$$
that induces $\Stwoone\dual(2)\cong \Stwoten$.
Composing $j$ with some element of $W(\Stwoten)\times\{\pm 1\}$,
we can assume that $j(w_M)$ is contained in
 $\NE(\Xtwoten)$ (Proposition~\ref{prop:nef}(2)).
The isomorphism $j_*:  \OG^+(\Stwoone)\isom \OG^+(\Stwoten)$ induced by $j$ is denoted by
$$
g\mapsto g\sprime.
$$
The $j(\RRR({\Lts, \iota}))\sphyp$-chamber $j(D(\Xtwoone))$ is fundamental,
and we have 
$$
\aut(j(D(\Xtwoone)))=\aut(D(\Xtwoone))\sprime.
$$
\begin{lemma}\label{lem:containtwo}
The set $j(\RRR(\Lts, \iota))$ contains $\RRR(\Stwoten)$.
Hence the $j(\RRR({\Lts, \iota}))\sphyp$-chamber $j(D(\Xtwoone))$
is contained in the $\RRR(\Stwoten)\sphyp$-chamber $\NE\spcirc(\Xtwoten)$.
\end{lemma}
\begin{proof}
It is enough to show that, if $v\in \Stwoone\dual$ satisfies 
$v^2=-1$, then $v\in \RRR(\Lts, \iota)$, that is, 
there exists $u\in \Tiota\dual$ such that $u^2=-1$ 
and that $u+v$ is contained in the submodule $\Lts$ of $\Stwoone\dual\oplus \Tiota\dual$.
By Nikulin~\cite[Proposition 1.4.1]{MR525944}, 
the submodule $\Lts/(\Stwoone\oplus \Tiota)$ of 
$(\Stwoone\dual\oplus \Tiota\dual)/(\Stwoone\oplus \Tiota)=A_{\Stwoone}\oplus A_{\Tiota}$
is the graph of an isomorphism
$$
q_{\Stwoone}\cong -q_{\Tiota}.
$$
Hence it is enough to show that,
for any $\bar{u}\in A_{\Tiota}$ with $q_{\Tiota}(\bar{u})=1$,
there exists $u\in \Tiota\dual$ such that $u^2=-1$ and $u\bmod \Tiota=\bar{u}$.
Since $\Tiota$ is a root lattice of type $D_4$,
we can confirm this fact by  direct computation.
The set of $(-1)$-vectors  in $\Tiota\dual$ consists of $24$ vectors,
and its image by the natural projection $\Tiota\dual\to A_{\Tiota}$ is the set of all non-zero elements
of $A_{\Tiota}\cong \F_2^2$.
\end{proof}
The set of  walls of $j(D(\Xtwoone))$ is equal to
\begin{eqnarray*}
\shortset{(j(e_i))\sperp}{i=1, \dots, 21} \;\cup\; 
\shortset{(j(f_i))\sperp}{i=1, \dots, 21} \;\cup\;  \\
\shortset{(j(c_I))\sperp}{\textrm{$I$ is a set of general six points}}.
\end{eqnarray*}
%
Note that the $21+21$ vectors $j(e_i)$ and $j(f_i)$ are of norm $-4$ and the $168$ vectors $j(c_I)$ are of norm $-2$.
Note also that neither $(j(e_i))\sperp$ nor $(j(f_i))\sperp$ are contained in $\RRR(\Stwoten)\sphyp$,
because  there are no rational numbers $\lambda$ such that $(-4)\lambda^2=-2$.
By Proposition~\ref{prop:FFFGGG}, Theorem~\ref{thm:char2sigma1} and Lemma~\ref{lem:containtwo},
we obtain the following:
\begin{theorem}\label{thm:char2sigma10}
The group $\aut(\NE(\Xtwoten))$ is generated by
$\PGL(3, \F_4)\sprime$, $\switch\sprime$,  $\Frob\sprime$, 
$s_{e_i}\sprime$ and $s_{f_i}\sprime$.
\end{theorem}
\subsection{The group $\aut(\NE(\Xthreeten))$}\label{subsec:AutNEchar3}
By the same argument as above,
we obtain  a set of generators of $\aut(\NE(\Xthreeten))$
from the result of Kondo and Shimada~\cite{KondoShimada}.
We consider the Fermat quartic surface
$$
X_{\FQ}: x_0^4+x_1^4+x_2^4+x_3^4=0
$$ 
in characteristic $3$.
Then $X_{\FQ}$ is isomorphic to $X_{3,1}$.
The surface $X_{\FQ}$ contains $112$ lines,
and their classes $l_1, \dots, l_{112}$ span $\Sthreeone$.
We denote by $h_{\FQ}\in \Sthreeone$ the class of a hyperplane section of $X_{\FQ}$.
\par
\medskip
There exists a primitive embedding $\iota : S_{3,1}\inj \Lts$,
which is   unique up to $\OG(\Lts)$.
The orthogonal complement 
$\Tiota$  is isomorphic to the root lattice of type $2A_2$, and 
hence satisfies the hypothesis~\eqref{eq:hypothesisT}.
We  calculated an $\RRR(\Lts, \iota)\sphyp$-chamber $D(\Xthreeone)$ 
that contains $h_{\FQ}$ in its interior, 
and found 
$$
\textrm{
$648$  vectors $u_j\in S_{3,1}\dual$ of norm $-4/3$, \;\;and \;
$5184$   vectors $w_k\in S_{3,1}\dual$ of norm $-2/3$}
$$
such that 
the walls of $D(\Xthreeone)$ consist of  the $112$ hyperplanes  $(l_i)\sperp$,
the $648$ hyperplanes  $(u_j)\sperp$ and the $5184$ hyperplanes $(w_k)\sperp$.
Note that the $\RRR(\Lts, \iota)\sphyp$-chamber $D(\Xthreeone)$ 
is contained in the $\RRR(\Sthreeone)\sphyp$-chamber $\NE\spcirc(\Xthreeone)$,
because $h_{\FQ}\in D(\Xthreeone)\spcirc$.
Moreover, since $28\, h_{\FQ}=\sum l_i$,  the following holds:
\begin{proposition}
An isometry $g\in \OG^+(\Sthreeone)$  belongs to   $\aut(D(\Xthreeone))$ if and only
if $h_{\FQ}^g=h_{\FQ}$.
\end{proposition}
%
%
%
%
%
We have the following elements  in $\Aut(X_{3,1})$ and $\OG^+(S_{3,1})$.
Note that, 
for a polarization $h\in \Sthreeone$ of degree $2$,
we have the deck transformation   $\tau(h)  \in \Aut(\Xthreeone)$ of the generically  finite morphism
$\Xthreeone\to \P^2$ of degree $2$ induced by the the complete linear system associated with $h$.
\begin{itemize}
\item
The subgroup $\PGU(4, \F_9)$ of $\PGL(4, k)=\Aut(\P^3)$ acts on $X_{\FQ}$.
Its action on $\Sthreeone$ preserves 
$D(X_{3,1})$.
Moreover, the action of 
$\PGU(4, \F_9)$ on $\Sthreeone\dual$  is transitive 
on each of the set of $112$ vectors $l_i$,
the set of $648$ vectors $u_j$ and the set of $5184$ vectors $w_k$.
\item
There exists a polarization $h_{648}\in \Sthreeone$ of degree $2$ such that
the deck transformation  $\tau(h_{648})\in \Aut(\Xthreeone)$  
maps $D(X_{3,1})$ to an $\RRR(\Lts, \iota)\sphyp$-chamber adjacent to $D(X_{3,1})$
along one of the $648$ walls $(u_j)\sperp$.
\item
There exists a polarization $h_{5184}\in \Sthreeone$ of degree $2$ such that
the deck transformation  $\tau(h_{5184})\in \Aut(\Xthreeone)$  
maps $D(X_{3,1})$ to an $\RRR(\Lts, \iota)\sphyp$-chamber adjacent to $D(X_{3,1})$
along one of the $5184$ walls $(w_k)\sperp$.
\item
The Frobenius action of $\Gal(\F_{9}/\F_3)$ on $X_{\FQ}$
gives rise to an element  $\Frob\in \aut(D(X_{3,1}))$ of order $2$.
\item
We have 
the reflections $s_{l_i}$ with respect  to the classes $l_i$ of the $112$ lines on $X_{\FQ}$.
\end{itemize}
\erase{
\begin{remark}
The actions of the involutions $\tau(h_{648})$ and $\tau(h_{5184})$ 
on $\Sthreeone$ are \emph{not} reflections.
\end{remark}
}
%
%
Thus $D(X_{3,1})$ is fundamental, and hence
we have the following:
\begin{theorem}[\cite{KondoShimada}] \label{thm:char3sigma1}
{\rm (1)}
The projective automorphism group $\Aut(X,  h_{\FQ})$ of  the Fermat quartic surface 
$X_{\FQ}\subset \P^3$ is equal to 
$\PGU(4, \F_9)$.
\par
{\rm (2)}
The group $\aut(D(X_{3,1}))$ is generated by $\Aut(X,  h_{\FQ})$ and $\Frob$.
\par
{\rm (3)}
The automorphism group $\Aut(X_{3,1})$ is generated by
$\Aut(X,  h_{\FQ})$ and the two involutions $\tau(h_{648})$ and $\tau(h_{5184})$.
\par
{\rm (4)}
The group $\aut(\NE(X_{3,1}))$ is generated by
$\Aut(\Xthreeone)$ and $\Frob$.
\par
{\rm (5)}
The group $\OG^+(S_{3,1})$ is generated by
$\aut(\NE(X_{3,1}))$ and the $112$ reflections $s_{l_i}$.
\end{theorem}
By Corollary~\ref{cor:j},
we have an embedding 
$$
j : \Sthreeone\inj \Sthreeten
$$
 that induces $\Sthreeone\dual(3)\cong \Sthreeten$.
By~Proposition~\ref{prop:nef}(2),
we can assume that $j(h_{\FQ})$ is contained in
 $\NE(\Xthreeten)$.
The isomorphism $j_*: \OG^+(\Sthreeone)\isom \OG^+(\Sthreeten)$ induced by $j$ is denoted by
$g\mapsto g\sprime$.
The $j(\RRR(\Lts, \iota))\sphyp$-chamber $j(D(\Xthreeone))$ is fundamental,
and 
$\aut(j(D(\Xthreeone)))$
is equal to $\aut(D(\Xthreeone))\sprime$.
\begin{lemma}\label{lem:containthree}
The set $j(\RRR(\Lts, \iota))$ contains $\RRR(\Sthreeten)$.
Hence the $j(\RRR(\Lts, \iota))\sphyp$-chamber $j(D(\Xthreeone))$
is contained in the $\RRR(\Sthreeten)\sphyp$-chamber $\NE\spcirc(\Xthreeten)$.
\end{lemma}
\begin{proof}
It is enough to show that, if $v\in \Sthreeone\dual$ satisfies 
$v^2=-2/3$, then 
there exists $u\in \Tiota\dual$ such that $u^2=-4/3$ 
and that $u+v$ is contained in $\Lts\subset \Sthreeone\dual\oplus \Tiota\dual$.
For this, it suffices  to prove that,
for any $\bar{u}\in A_{\Tiota}$ with $q_{\Tiota}(\bar{u})=-4/3$,
there exists $u\in \Tiota\dual$ such that $u^2=-4/3$ and $u\bmod \Tiota=\bar{u}$.
Since $\Tiota$ is a root lattice of type $2A_2$,
we can confirm this fact by  direct computation.
\end{proof}
The set of walls of $j(D(\Xthreeone))$ is equal to
\begin{eqnarray*}
\shortset{(j(l_i))\sperp}{i=1, \dots, 112} \;\cup\; 
\shortset{(j(u_j))\sperp}{j=1, \dots, 648} \;\cup\;  \hfill \\
\hfill  \shortset{(j(w_k))\sperp}{k=1, \dots, 5184}.
\end{eqnarray*}
Note that 
the  vectors $j(l_i)$ are of norm $-6$,
the  vectors $j(u_j)$ are  of norm $-4$,
and the  vectors $j(w_k)$ are of norm $-2$.
Note also that neither $(j(l_i))\sperp$ nor $(j(u_j))\sperp$ are contained in $\RRR(\Sthreeten)\sphyp$.
By Proposition~\ref{prop:FFFGGG}, Theorem~\ref{thm:char3sigma1} and Lemma~\ref{lem:containthree},
we obtain the following:
\begin{theorem}\label{thm:char3sigma10}
The group $\aut(\NE(\Xthreeten))$ is generated by
$\PGU(4, \F_9)\sprime$,   $\Frob\sprime$, 
$s_{l_i}\sprime$ and $\tau({h_{648}})\sprime$.
\end{theorem}
%
%
\section{Torelli theorem for supersingular $K3$ surfaces}\label{sec:Torelli}
We review  the theory of period mapping and 
Torelli theorem  for supersingular $K3$ surfaces 
in odd characteristics  by Ogus~\cite{MR563467}, \cite{MR717616}.
Throughout this section,
we assume that $p$ is odd.
\par
\medskip
We summarize results on  quadratic spaces over finite fields.
Let $\F_q$ be a finite extension of $\F_p$.
There exist exactly two isomorphism classes of non-degenerate quadratic forms
in $2\sigma$ variables
$x_1, \dots, x_{2\sigma}$ over $\F_q$.
They are represented by
\begin{eqnarray}\label{eq:fplus}
f_+&:=& x_1 x_2 + \dots +x_{2\sigma-1} x_{2\sigma}, \quand \\ \label{eq:fminus}
f_-&:=& x_1^2+c x_1 x_2 +x_2^2 + x_3 x_4 +   \dots +x_{2\sigma-1} x_{2\sigma},
\end{eqnarray}
where $c$ is an element of $\F_q$ such that $t^2+c t +1\in \F_q[t]$ is irreducible.
The quadratic form $f_+$ (resp.~$f_-$) is called
\emph{neutral} (resp.~\emph{non-neutral}).
We denote by $\OG(\F_q^{2\sigma}, f_{\epsilon})$
the group of the self-isometries of the quadratic space $(\F_q^{2\sigma}, f_{\epsilon})$.
\erase{
The  group $\OG(\F_q^{2\sigma}, f_{\epsilon})$
of the self-isometries of the quadratic space $(\F_q^{2\sigma}, f_{\epsilon})$, 
where $\epsilon=\pm 1$, 
is of order 
\begin{equation*}\label{eq:order}
2\,q^{\sigma(\sigma-1)}(q^{\sigma}-\epsilon)\prod_{i=1}^{\sigma-1}  (q^{2i}-1).
\end{equation*}
}
\par
\medskip
Let $N$ be  an even hyperbolic $p$-elementary lattice of rank $22$ with discriminant $-p^{2\sigma}$.
We define a quadratic space $(N_0, q_0)$ over $\F_p$ by~\eqref{eq:N0q0}.
It is known that $q_0$ is non-degenerate and \emph{non-neutral}.
We denote by 
$\OG(N_0, q_0)$
the group  of the self-isometries of  $(N_0, q_0)$.
Note that the scalar multiplications in $\OG(N_0, q_0)$ are only $\pm 1$.
%
Let $k$ be a   field of characteristic $p$.
We put
$$
\varphi:=\id_{N_0}\tensor F_k : N_0\tensor k \to N_0\tensor k,
$$
where $F_k$ is the Frobenius map of $k$.
\begin{definition}\label{def:characteristic}
A subspace $K$ of $N_0\tensor k$ with $\dim K=\sigma$ is said to be a~\emph{characteristic subspace}
of $(N_0, q_0)$
if $K$ is totally isotropic with respect to the quadratic form $q_0\tensor k$
and $\dim (K\cap \varphi (K))=\sigma-1$ holds.
\end{definition}
Suppose that $k$ is algebraically closed.
Let $X$ be a supersingular $K3$ surface with Artin invariant $\sigma$
defined over $k$.
An isomorphism
$$
\eta: N\isom\,  S_X
$$
of lattices 
is called a \emph{marking} of $X$.
We fix a marking $\eta$ of $X$.
The  composite of the marking $\eta$ and the Chern class map $S_X \to H^2_{\DR}(X/k)$ defines
a linear homomorphism
$$
\bar{\eta}: N\tensor k \to H^2_{\DR}(X/k).
$$
It is known that $\Ker\bar\eta$ is contained in $N_0\tensor k$,
and is totally isotropic with respect to $q_0\tensor k$.
We put 
$$
K_{(X, \eta)}:=\varphi\inv (\Ker\bar\eta),
$$
and call $K_{(X, \eta)}$ the \emph{period} of the marked supersingular $K3$ surface $(X, \eta)$.
Then it is proved by Ogus~\cite{MR563467}, \cite{MR717616} that 
$K_{(X, \eta)}$ is a characteristic subspace of $(N_0, q_0)$.
We denote by  $\eta^*: \OG(S_X) \isom \OG(N)$ 
 the isomorphism induced by the marking $\eta$,
and let
$$
\bar\eta^*: \OG(S_X) \to \OG(N_0, q_0)
$$
be the composite of $\eta^*$ with the natural homomorphism $\OG(N)\to \OG(N_0, q_0)$.
As a corollary of Torelli theorem by Ogus~\cite[Corollary of Theorem II${}^{\prime\prime}$]{MR717616},
 we have the following:
 \begin{corollary}\label{cor:finiteindex}
 Let $\eta$ be a marking of $X$.
 Then, 
 as a subgroup of $\OG^+(S_X)$,
 the automorphism group $\Aut(X)$ of $X$ is equal to
 $$
 \set{g\in \aut(\NE(X))}{ K_{(X, \eta)}^{\bar{\eta}^*(g)} = K_{(X, \eta)}}.
 $$
 In particular, the index of $\Aut(X)$ in $\aut(\NE(X))$ is at most
 $|\OG(N_0, q_0)/\{\pm 1\}|$.
 \end{corollary}
 Since $\aut(\NE(X))$ is infinite,
 we obtain the following:
 \begin{corollary}\label{cor:infinite}
 The automorphism group $\Aut(X)$  is infinite.
\end{corollary}
\erase{
 \begin{remark}\label{rem:largeorder}
  When $p=3$ and $\sigma=1$, the  group  $\OG(N_0, q_0)$ is of order $8$,
  while the index of $\Aut(\Xthreeone)$ in $\aut(\NE(\Xthreeone))$ is $2$
  by Theorem~\ref{thm:char3sigma1}.
  \end{remark}
  }
 \begin{definition}\label{def:generic}
 We say that  a supersingular $K3$ surface $X$ with Artin invariant $\sigma$
is \emph{generic}
if there exists a marking $\eta$ for $X$ such that the subgroup
\begin{equation}\label{eq:stabK}
\set{\gamma\in \OG(N_0, q_0)}{ K_{(X, \eta)}^{\gamma} = K_{(X, \eta)}}
\end{equation}
of $\OG(N_0, q_0)$ consists of only scalar multiplications $\pm 1$.
 \end{definition}
If $X$ is generic, then the subgroup~\eqref{eq:stabK} consists of only scalar multiplications 
for any marking $\eta$.
The existence of generic   supersingular $K3$ surfaces  with Artin invariant $>1$
(Theorem~\ref{thm:existenceofgeneric})
is proved in the next section.
\par\medskip
Recall that $A_{S_X}$ is the discriminant group of $S_X$,
and  $q_{S_X} : A_{S_X}\to \Q/2\Z$ is the  discriminant quadratic form.
We will regard $A_{S_X}$ as a $2\sigma$-dimensional vector space over $\F_p$.
Note that the image of $q_{S_X}$ is contained in $(2/p)\Z/2\Z$.
We define $\bar{q}_{S_X} : A_{S_X} \to \F_p$ by 
$$
\bar{q}_{S_X} (x \bmod S_X) := p \cdot q_{S_X} (x) \bmod p.
$$
Then we obtain a quadratic space $(A_{S_X}, \bar{q}_{S_X} )$ over $\F_p$.
Note that we can recover $q_{S_X}$ from $\bar{q}_{S_X}$.
We have natural homomorphisms 
\begin{equation}\label{eq:toPOG}
\OG(S_X) \to \OG(q_{S_X}) \cong  \OG (A_{S_X}, \bar{q}_{S_X})  \surj 
\POG  (A_{S_X}, \bar{q}_{S_X}):= \OG (A_{S_X}, \bar{q}_{S_X})/\{\pm 1\}.
\end{equation}
Let $\eta: N\dual\isom S_X\dual$ be the isomorphism induced by a marking $\eta$.
Then the map
$$
px \bmod pN \in N_0 \mapsto \eta (x) \bmod S_X \in A_{S_X}\qquad (x\in N\dual)
$$
induces an isomorphism of quadratic spaces from $(N_0, q_0)$ to $(A_{S_X}, \bar{q}_{S_X})$.
By Corollary~\ref{cor:finiteindex}, we obtain the following:
%
%
\begin{corollary}\label{cor:genericaut}
Suppose that $X$ is  generic.
Then $\Aut(X)$ is equal to the kernel of the homomorphism 
%
$$
\Phi: \aut(\NE(X)) \to \POG  (A_{S_X}, \bar{q}_{S_X})
$$
%
obtained by restricting~\eqref{eq:toPOG} to $\aut(\NE(X)) \subset \OG(S_X)$.
\end{corollary}
\begin{remark}\label{rem:heavy}
Suppose that $\Xthreeten$ is generic.
By the standard algorithm of combinatorial group theory, 
we can obtain a finite set of generators of $\Aut(\Xthreeten)$
from the generators of $\aut(\NE(\Xthreeten))$ given in Theorem~\ref{thm:char3sigma10}.
However, a naive application of the algorithm
would be inexecutable,
because, 
when $p=3$ and $\sigma=10$, the order of $\OG(N_0, q_0)$ is 
$$
   2^{36} \cdot   3 ^{90} \cdot   5^{6} \cdot   7 ^{3} \cdot   11^{2} \cdot 
  13^{3} \cdot 17  \cdot 19 
 \cdot 37  \cdot   41^{2} \cdot 61 \cdot  
   73\cdot     193\cdot     547\cdot   
  757\cdot     1093\cdot     1181 ,
  $$
which is about $7.886 \times 10^{90}$.
In fact, 
we can  show that
there exists a genus one fibration  $\phi\sprime$ on $\Xthreeten$
whose lattice equivalence class contains at least $6531840$ $\Aut$-equivalence classes, 
by giving the class $f_{\phi}$ of  a Jacobian fibration on $\Xthreeone=X_{\FQ}$ explicitly. 
\end{remark}
%
%
%
%
%
%
 \section{Existence of generic supersingular $K3$ surfaces}\label{sec:existenceofgeneric}
 We prove Theorem~\ref{thm:existenceofgeneric}.
 For the proof,
 we recall the construction by  Ogus~\cite{MR563467}
 of the  scheme $\MMM$ parameterizing characteristic subspaces of 
 the $2\sigma$-dimensional quadratic space $(N_0, q_0)$ over $\F_p$. 
This scheme $\MMM$  plays the role of the period domain for supersingular $K3$ surfaces.
We continue to assume that $p$ is odd.
 \par
 \medskip
 Let  $\Grass(\nu, N_0)$ denote the Grassmannian variety of $\nu$-dimensional subspaces of $N_0$,
 and let $\Isot(\nu, q_0)$ be the subscheme of $\Grass(\nu, N_0)$ parameterizing $\nu$-dimensional 
 totally isotropic subspaces of $(N_0, q_0)$.
 We put
$$
\Gen:=\Isot(\sigma, q_0), 
$$
where $\Gen$ is for ``generatrix".
Note that  $\Isot(\nu, q_0)$ is defined over $\F_p$ for any $\nu$.
Let $k$ be a field of characteristic $p$.
For a subspace $L$ of $N_0\tensor k$ with dimension $\nu$,
 we denote by $[L]$ the $k$-valued point of $\Grass(\nu, N_0)$ corresponding to $L$.
We then have the following:
\begin{enumerate}
\item If $\nu<\sigma$, then  $\Isot(\nu, q_0)$ is geometrically connected.
\item The scheme $\Gen\tensor \F_{p^2}$ has two connected components $\Gen_+$ and $\Gen_-$,
each of which is geometrically connected.
Since $q_0$ is non-neutral,
the action of $\Gal(\F_{p^2}/\F_{p})$ interchanges the two connected components.
\item Let $K$ and $K\sprime$ be two $\sigma$-dimensional 
 totally isotropic subspaces of $(N_0, q_0)\tensor k$.
 Suppose that  $\dim (K\cap K\sprime)=\sigma-1$.
 Then the $k$-valued points $[K]$ and $[K\sprime]$
 belong to distinct connected components of $\Gen$.
 \item 
 Suppose that 
 $k$ is algebraically closed. 
 Then, for each $k$-valued point $[L]$ of the scheme $\Isot(\sigma-1, q_0)$,
  there exist exactly two $\sigma$-dimensional 
 totally isotropic subspaces  of $(N_0, q_0)\tensor k$
 that contain $L$.
\item
Let $P$ be the subscheme of $\Gen\times \Gen$ parameterizing 
pairs $(K, K\sprime)$ such that $\dim (K\cap K\sprime)=\sigma-1$.
Then the scheme $P\tensor\F_{p^2}$ has two connected components,
each of which is isomorphic to   $\Isot(\sigma-1, q_0)$ over $\F_{p^2}$.
The action of $\Gal(\F_{p^2}/\F_{p})$ interchanges the two connected components.
\end{enumerate}
Consider the graph 
$$
\id \times \varphi: \Gen\to \Gen\times \Gen
$$
of the  Frobenius morphism $\Gen\to \Gen$ given by $K\mapsto\varphi(K)$.
The subscheme $\MMM$ of $\Gen$ that parametrizes the characteristic subspaces of $(N_0, q_0)$
is defined by the fiber product 
$$
\begin{array}{cccc}
\MMM && \inj & \Gen \\
\downarrow && \hbox{\small $\fiberproduct$} & \downarrow \rlap{\small$\id \times \varphi$} \\
P &&\inj & \Gen\times\Gen.
\end{array}
$$
Ogus~\cite{MR563467} proved the following:
\begin{theorem}
The scheme $\MMM$ defined over $\F_p$ is  smooth and projective  of dimension $\sigma-1$.
The scheme  $\MMM\tensor \F_{p^2}$ has two connected components $\MMM_+=\MMM\cap \Gen_+$ and $\MMM_-=\MMM\cap \Gen_-$,
each of which is geometrically connected.
The action of $\Gal(\F_{p^2}/\F_{p})$ interchanges $\MMM_+$ and $\MMM_-$.
\end{theorem}
 \begin{proof}[Proof of Theorem~\ref{thm:existenceofgeneric}]
 Let $\kappa$ be an algebraic closure of the function field of 
 the scheme $\MMM_+$ over $\F_{p^2}$,
 and let $[K_{\kappa}]$ be the geometric generic point of $\MMM_+$. 
 By the surjectivity of the period mapping  for supersingular $K3$ surfaces
 (Ogus~\cite[Theorem III${}^{\prime\prime}$]{MR717616}), 
 there exist a supersingular $K3$ surface $X$ of Artin invariant $\sigma$ defined over $\kappa$
 and a marking $\eta: N\isom S_X$ such that $K_{(X, \eta)}=K_{\kappa}$.
 We prove that this $X$ is generic,
 that is, 
 $$
 G_{\kappa}:=\set{\gamma\in \OG(N_0, q_0)}{ K_{\kappa}^{\gamma} = K_{\kappa}}
 $$
 is equal to $\{\pm 1\}$.
 Note that
 the closure of the point $[K_{\kappa}]$ coincides with   $\MMM_+$. 
 Therefore we have the following:
 \textrm{If a field $k$ contains $\F_{p^2}$, then the action of $G_{\kappa}$ leaves  $K$  invariant
for any $k$-valued point $[K]$ of $\MMM_+$}.

 \par
 \medskip
 Suppose that $\sigma\ge 3$.
 Let $u$ be an arbitrary  non-zero isotropic vector of $N_0$.
 We will prove that $u$ is an eigenvector of $G_{\kappa}$.
 Let 
 $$
 b_0: N_0\times N_0\to \F_p
 $$
  denote the symmetric bilinear form obtained from $q_0$.
 There exists a vector $v\in N_0$ such that $q_0(v)=0$ and $b_0(u, v)=1$,
 and hence $(N_0, q_0)$ has an orthogonal direct-sum decomposition
 $$
 N_0=U\sperp \oplus U,
 $$
 where $U$ is the subspace spanned by $u$ and $v$.
 Repeating this procedure and noting that $q_0$ is non-neutral, 
we obtain  a basis $a_1, \dots, a_{2\sigma}$ of $N_0$
 with $u=a_{2\sigma}$ such that
 $q_0(x_1a_1+\cdots+x_{2\sigma}a_{2\sigma})$ is equal to 
 the quadratic  polynomial $f_-$ in~\eqref{eq:fminus}.
 Let $\alpha$ and $\bar{\alpha}=\alpha^p$ be the roots in $\F_{p^2}$ of 
 the irreducible polynomial $t^2+c t +1 \in \F_p[t]$.
 We consider the basis
\begin{equation}\label{eq:basisbs}
 \renewcommand{\arraystretch}{1.2}
\begin{array}{ll}
 b_1^{(-1)}:=\alpha a_1+a_2, &
 b_1^{(1)}:=\bar \alpha a_1+a_2, \quand \\
 b_i^{(-1)}:=a_{2i-1}, &
 b_i^{(1)}:=a_{2i}\qquad
  (i=2, \dots, \sigma)
 \end{array}
\end{equation}
of $N_0\tensor\F_{p^2}$.
Note that each $b_{i}^{(\pm 1)}$ is
isotropic, and that
$$
b_0(b_i^{(\alpha)}, b_j^{(\beta)})=0\;\;\;\textrm{if $i\ne j$},
\qquad
b_0(b_i^{(1)}, b_i^{(-1)})=\begin{cases}
(4-c^2)/2  &\textrm{if $i=1$}, \\
1/2 & \textrm{if $i\ge 2$}.
\end{cases}
$$
We put
$$
\EEE:=\{1, -1\}^{\sigma}.
$$
For $e=(\varepsilon_1,\dots, \varepsilon_\sigma)\in \EEE$,
we denote by $K_e$ the linear subspace of $N_0\tensor\F_{p^2}$ spanned by
$$
b_1^{(\varepsilon_1)},\dots, b_\sigma^{(\varepsilon_\sigma)}.
$$ 
It is obvious that $K_e$ is isotropic.
Moreover, since
$$
\varphi(b_1^{(\varepsilon)})=b_1^{(-\varepsilon)}
\quand
\varphi(b_i^{(\varepsilon)})=b_i^{(\varepsilon)}\;\;\textrm{if $i\ge 2$}, 
$$
 we have $\dim (K_e\cap \varphi(K_e))=\sigma-1$.
 Therefore $K_e$ is a characteristic subspace of $(N_0, q_0)$.
 Suppose that $e$ and $e\sprime\in \EEE$ differ only at one component.
 Then we have $\dim (K_e\cap K_{e\sprime})=\sigma-1$,
 and hence the $\F_{p^2}$-valued points $[K_e]$ and $[K_{e\sprime}]$ of $\MMM$ belong to
 distinct connected components.
 We put
 $$
 \EEE_+:=\set{e\in \EEE}{\textrm{the number of $-1$ in $e$ is even}},
 \quad
 \mathbf{1}:=(1,\dots, 1)\in \EEE_+.
 $$
 Interchanging $\alpha$ and $\bar\alpha$ if necessary,
 we can assume that $[K_{\mathbf{1}}]$ is an $\F_{p^2}$-valued point of $\MMM_+$,
 and hence $[K_e]$ is an  $\F_{p^2}$-valued point of $\MMM_+$ for any $e\in \EEE_+$.
It follows that  $K_e$ is invariant under the action of $G_{\kappa}$ for any $e\in \EEE_+$.
Let $b_{i}^{(\alpha)}$ be an arbitrary element among the basis~\eqref{eq:basisbs}.
Recall that we have assumed  $\sigma\ge 3$.
Therefore, for each element $b_{j}^{(\beta)}$ among the basis~\eqref{eq:basisbs}
that is distinct from $b_{i}^{(\alpha)}$, 
there exists $e(j, \beta)=(\varepsilon_1,\dots, \varepsilon_\sigma)\in \EEE_+$ such that
$\varepsilon_i=\alpha$ and $\varepsilon_j\ne \beta$.
Since
$$
\bigcap_{(j, \beta)\ne (i, \alpha)} K_{e(j, \beta)}=\gen{b_{i}^{(\alpha)}}
$$
is invariant under the action of $G_{\kappa}$,
we see that $b_{i}^{(\alpha)}$ is an  eigenvector of  $G_{\kappa}$.
In particular,
the  isotropic vector $u=a_{2\sigma}=b_{\sigma}^{(1)}$  given at the beginning 
is  an eigenvector of  $G_{\kappa}$.
\par
Let 
$$
\lambda_{i}^{(\alpha)}: G_{\kappa}\to \F_{p^2}\sptimes
$$
 be the character defined by  $b_{i}^{(\alpha)}$.
%
Suppose that $i, j\ge 2$ and $i\ne j$.
Then $b_{i}^{(\alpha)}+b_{j}^{(\beta)}$ is an isotropic vector of $N_0$
for any choice of $\alpha, \beta\in \{\pm 1\}$,
and hence is an eigenvector of $G_{\kappa}$.
Therefore we have 
\begin{equation}\label{eq:lambda2}
\lambda_{i}^{(\alpha)}=\lambda_{j}^{(\beta)}\qquad\textrm{if $i, j\ge 2$ and $i\ne j$.}
\end{equation}
%
Since the cardinality of $\shortset{x^2}{x\in \F_p}$ is $(p+1)/2$, 
there exist $\xi, \eta\in \F_p$ such that
$$
(4-c^2)+\xi^2+\eta^2=0.
$$
Then 
$$
b_{1}^{(1)}+b_{1}^{(-1)}+\xi (b_{2}^{(1)}+b_{2}^{(-1)})+\eta (b_{3}^{(1)}+b_{3}^{(-1)})
$$ 
is also an isotropic vector of $N_0$, and hence is  an eigenvector of $G_{\kappa}$.
Therefore we have 
\begin{equation}\label{eq:lambda3}
\lambda_{1}^{(1)}=\lambda_{1}^{(-1)}=\lambda_{2}^{(1)}=\lambda_{2}^{(-1)}
\quad
\textrm{or}
\quad
\lambda_{1}^{(1)}=\lambda_{1}^{(-1)}=\lambda_{3}^{(1)}=\lambda_{3}^{(-1)}.
\end{equation}
Combining~\eqref{eq:lambda2} and~\eqref{eq:lambda3},
we see that all the characters $\lambda_{i}^{(\alpha)}$ are equal to each other.
Thus $G_{\kappa}$ consists of only scalar multiplications.
\par
\medskip
Suppose that $\sigma=2$.
In this case,
the scheme $\MMM$ coincides with  $\Isot(2, q_0)$,
which is the scheme parametrizing lines on 
the smooth quadratic surface $Q_0=\{q_0=0\}$ in 
the projective space $\P_*N_0=\Grass(1, N_0)$.
Hence  $\MMM_+$ and $\MMM_-$ 
correspond to
the two rulings of $Q_0$.
Let $g$ be an element of $G_{\kappa}$.
Then $g$ leaves every line in the ruling of $Q_0$
corresponding to $\MMM_+$ invariant.
Since $g$ is defined over $\F_p$ and $\Gal(\F_{p^2}/\F_{p})$ interchanges $\MMM_+$ and $\MMM_-$,
we see that $g$ also leaves every line in the other ruling of $Q_0$ invariant.
Therefore $g$ fixes every point of $Q_0$,
and hence every point of $\P_*N_0$.
\end{proof}
\bibliographystyle{plain}
\def\cftil#1{\ifmmode\setbox7\hbox{$\accent"5E#1$}\else
  \setbox7\hbox{\accent"5E#1}\penalty 10000\relax\fi\raise 1\ht7
  \hbox{\lower1.15ex\hbox to 1\wd7{\hss\accent"7E\hss}}\penalty 10000
  \hskip-1\wd7\penalty 10000\box7} \def\cprime{$'$} \def\cprime{$'$}
  \def\cprime{$'$} \def\cprime{$'$}

\end{document}